\documentclass[a4paper,11pt]{article}

\usepackage{a4wide}
\usepackage{hyperref} 

\usepackage{amssymb,amsmath,amsthm}
\usepackage[english,french]{babel}

\usepackage{todonotes}
\usepackage{verbatim} 
\usepackage{enumerate}
\usepackage{mathtools}
\usepackage[normalem]{ulem}
\usepackage{framed}  

\usepackage{bbm}

\usepackage{tikz}
\usetikzlibrary{matrix}

\usepackage[backend=bibtex,style=numeric-comp,url=false,isbn=false,doi=false,maxbibnames=10]{biblatex}
\usepackage[utf8]{inputenc}
\usepackage[T1]{fontenc}
\usepackage{lmodern}

\newcommand{\cA}{\mathcal{A}}
\newcommand{\cB}{\mathcal{B}}
\newcommand{\cC}{\mathcal{C}}
\newcommand{\cD}{\mathcal{D}}

\newcommand{\cG}{\mathcal{G}}
\newcommand{\cH}{\mathcal{H}}

\newcommand{\cL}{\mathcal{L}}

\newcommand{\cP}{\mathcal{P}}

\newcommand{\bR}{\mathbb{R}}

\newcommand{\bfF}{\mathbf{F}}

\newcommand{\PR}{\mathbb{P}}

\newcommand{\bONE}{\mathbbm{1}}

\newcommand{\dd}{ \mathrm{d}}

\DeclareMathOperator*{\argmin}{argmin}

\DeclareMathOperator{\dv}{div}

\renewcommand{\epsilon}{\varepsilon}

\newcommand{\vn}[1]{\left| \! \left| #1\right| \! \right|}

\newcommand{\ip}[2]{\langle #1,#2\rangle}

\numberwithin{equation}{section}

\newtheorem{theorem}{Theorem}[section]
\newtheorem{lemma}[theorem]{Lemma}
\newtheorem{proposition}[theorem]{Proposition}
\newtheorem{corollary}[theorem]{Corollary}

\theoremstyle{definition}
\newtheorem{definition}[theorem]{Definition}

\newtheorem{remark}[theorem]{Remark}

\newtheorem{condition}[theorem]{Condition}

\makeatletter  
\definecolor{MyGray}{rgb}{0.95,0.95,0.95}
\definecolor{shadecolor}{rgb}{0.95,0.95,0.95}
\newdimen\grayboxwidth
\newdimen\grayboxmargin
\newdimen\grayboxinset
\newenvironment{example}%
  {
  \begin{MakeFramed}{\FrameSep1cm\advance\hsize-\width \FrameRestore}%
  \refstepcounter{theorem}\noindent\textbf{Example~\thetheorem.}}%
  {\end{MakeFramed}}
\makeatother

\let\aux a

\begin{document}

\title{Fluctuation symmetry leads to GENERIC equations with non-quadratic dissipation}


\author{
	\renewcommand{\thefootnote}{\arabic{footnote}}
	Richard C. Kraaij\footnotemark[1],
	\renewcommand{\thefootnote}{\arabic{footnote}}
	Alexandre Lazarescu\footnotemark[2],
	\renewcommand{\thefootnote}{\arabic{footnote}}
	Christian Maes\footnotemark[3],
	\renewcommand{\thefootnote}{\arabic{footnote}}
	Mark A. Peletier\footnotemark[4]
}

\footnotetext[1]{
	Fakultät für Mathematik, Ruhr-Universität Bochum
}
\footnotetext[2]{
	Centre de Physique Théorique, École Polytechnique at Paris
}
\footnotetext[3]{
	Instituut voor Theoretische Fysica, KU Leuven
}
\footnotetext[4]{
Institute for Complex Molecular Systems, TU Eindhoven
}



\selectlanguage{english} 
\maketitle

\begin{abstract}
	
	We develop a formalism to discuss the properties of \emph{GENERIC} systems in terms of corresponding \emph{Hamiltonians} that appear in the characterization of large-deviation limits. We demonstrate how the GENERIC structure naturally arises from a certain symmetry in the Hamiltonian, which extends earlier work that has connected the large-deviation behaviour of reversible stochastic processes to the gradient-flow structure of their deterministic limit. Natural examples of application include particle systems with inertia.

\end{abstract}

%




\tableofcontents

\section{Introduction}

For a long time there has been a `folk theorem' in the probability community that states the following: if a stochastic particle system is reversible, and if the particle system converges to a deterministic limit as the number of particles tends to infinity, then this deterministic limit is a gradient flow. Such a property has wide implications: it implies the presence of a Lyapunov function, it identifies a geometry of the underlying space that is generated by the noise, it constrains the type of evolution that the limit equation can show, and it gives essential insight in both the limit equation and the connection between the limit equation and the underlying stochastic process.

For linear stochastic differential equations with Gaussian noise this `folk theorem' was shown by Onsager and Machlup~\cite{OnMa53}, and this property has also been identified for various specific examples~\cite{KOV89,KiOl90,BGL05,PeReVa14}. All these works, however,   are limited to the case of `quadratic' noise (we explain this below, cf.\ Remark \ref{remark:quadratic_gradient_flow}), and exclude for instance jump processes. The theorem was recently proven under assumptions of smoothness, but otherwise in full generality, covering both quadratic and non-quadratic noise,  by Mielke, Peletier, and Renger~\cite{MPR14}. Their result states that for any sequence of reversible stochastic processes with a deterministic limit and corresponding large-deviation principle, the large-deviation rate functional gives rise to a gradient-flow structure for the deterministic limit. 

Since that result is central to this paper we briefly describe the global argument. Pathwise large-deviation principles of sequences of stochastic processes, reversible or not, can be expressed in so-called Hamiltonians and their duals, called Lagrangians (see Section~\ref{section:mo_LDP}). If the processes are reversible, then the Hamiltonian and the Lagrangian possess a certain symmetry. This symmetry implies that the Lagrangian can be written in a form that can be recognised as a generalization of the classical gradient-flow concept. 

The aim of this paper is to do the same for a class of non-reversible processes: we connect a symmetry property of a class of non-reversible stochastic processes to a variational-evolutionary structure for a corresponding class of deterministic limiting equations. This class, known as \emph{GENERIC}, consists of evolution equations for an unknown $z\in Z$ of the type
\[
\dot z = L(z) \dd E(z) + \partial_\xi \Psi^*\bigl(z,-\tfrac12 \dd S(z)\bigr).
\]
Here $E,S:Z\to\bR$ are two functionals with derivatives $\dd S$ and $\dd E$, for each $z\in Z$ the operator $L(z)$  maps the derivative $dE(z)$ to a velocity vector, and $(z,\xi) \mapsto \Psi^*(z,\xi)$ is a non-negative functional. We discuss this structure in more detail in Section~\ref{subsec:GENERIC}. 

Gradient flows are special cases of GENERIC equations, in which the term $L(z)dE(z)$ is absent, and our point of view in this paper is that the appearance of the additional term $L(z)dE(z)$ in GENERIC can be connected to the breaking of the reversibility symmetry of the underlying stochastic processes. While we can not make this statement in full generality, we do make a number of steps in that direction, which form the novelty of this paper:
\begin{enumerate}[A.]
	\item 
	\label{innov:Ham-GENERIC-connection}
	As in the case of gradient flows, a central role is played by the Hamiltonians that arise in the large-deviation behaviour. The first step is to develop a formalism to express GENERIC systems in terms of {Hamiltonians}. 
	\item \label{innov:SP-Ham-connection}
	Using this formalism, and its interpretation in terms of large-deviation principles, we connect symmetry properties of the underlying stochastic processes with a notion we call \textit{pre-GENERIC}. 
	\item \label{innov:pre-GENERIC-to-GENERIC}
	We show that pre-GENERIC can always be extended to GENERIC by a trivial addition of an auxiliary energy variable, showing that pre-GENERIC captures the relevant structure of GENERIC.
	\item \label{innov:examples}
	We illustrate these ideas on two examples involving particles with inertia. One is the Vlasov-Fokker-Planck system; the other is known as the Andersen thermostat and combines deterministic drift with a jump process in momentum space. 
\end{enumerate}

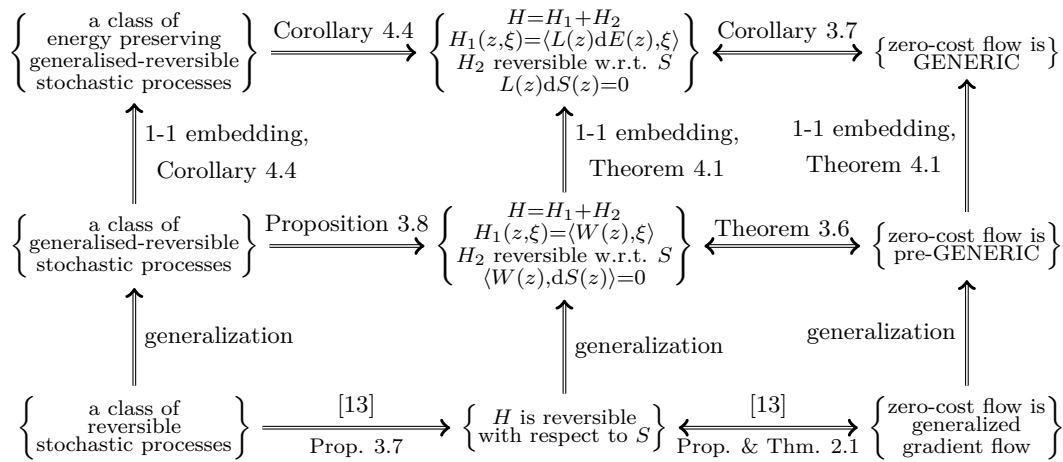
\begin{figure}[t]
	\begin{tikzpicture}
	\matrix (m) [matrix of math nodes,row sep=3em,column sep=5em,minimum width=2em]
	{
		\left\{\substack{\text{a class of} \\ \text{energy preserving} \\ \text{generalised-reversible} \\ \text{stochastic processes}} \right\} & 
		\left\{\substack{H = H_1 + H_2 \\ H_1(z,\xi) = \ip{L(z) \dd E(z)}{\xi} \\ H_2 \text{ reversible w.r.t. }S  \\ L(z) \dd S(z) = 0} \right\}& 
		\left\{ \substack{\text{zero-cost flow is} \\ \text{GENERIC}} \right\} \\	
		\left\{\substack{\text{a class of} \\ \text{generalised-reversible} \\ \text{stochastic processes}} \right\}& 
		\left\{\substack{H = H_1 + H_2 \\ H_1(z,\xi) = \ip{W(z)}{\xi} \\ H_2 \text{ reversible w.r.t.} \ S \\ \ip{W(z)}{\dd S(z)} = 0} \right\} & 
		\left\{ \substack{\text{zero-cost flow is} \\ \text{pre-GENERIC}} \right\} \\
		\left\{\substack{\text{a class of} \\ \text{reversible} \\ \text{stochastic processes}} \right\} & 
		\left\{\substack{H \text{ is reversible} \\ \text{with respect to } S} \right\}& 
		\left\{\substack{\text{zero-cost flow is} \\ \text{generalized} \\ \text{gradient flow}} \right\} \\
	};
	\draw[double,->] (m-1-1) -- (m-1-2) node [above,midway] {\footnotesize Corollary \ref{corollary:trivial_embedding_processes_pre_GENERIC_to_GENERIC}};
	\draw[double,<->] (m-1-2) -- (m-1-3) node [above,midway] {\footnotesize  Corollary \ref{corollary:GENERIC_iff_generalized_fluctuation_symmetry_orthogonality}};
	\draw[double,->] (m-2-1) -- (m-2-2) node [above,midway] {\footnotesize Proposition \ref{proposition:pre-GENERIC_from_LDP}};
	\draw[double,<->] (m-2-2) -- (m-2-3) node [above,midway] {\footnotesize  Theorem \ref{theorem:pre_GENERIC_iff_generalized_fluctuation_symmetry}};
	\draw[double,->] (m-3-1) -- (m-3-2) node [above,midway] {\footnotesize  \cite{MPR14}} node [below,midway] {\scriptsize Prop. 3.7};
	\draw[double,<->] (m-3-2) -- (m-3-3) node [above,midway] {\footnotesize  \cite{MPR14}} node [below,midway] {\scriptsize Prop. \& Thm. 2.1};

	\draw[double,<-] (m-1-1) -- (m-2-1) node [right,midway,align=center] {\footnotesize 1-1 embedding, \\ \footnotesize Corollary \ref{corollary:trivial_embedding_processes_pre_GENERIC_to_GENERIC}};
	\draw[double,<-] (m-2-1) -- (m-3-1) node [right,midway] {\footnotesize generalization};
	\draw[double,<-] (m-1-2) -- (m-2-2) node [right,midway,align=center] {\footnotesize 1-1 embedding, \\ \footnotesize Theorem \ref{theorem:pre_GENERIC_to_GENERIC}};
	\draw[double,<-] (m-2-2) -- (m-3-2) node [right,midway] {\footnotesize generalization};
	\draw[double,<-] (m-1-3) -- (m-2-3) node [left,midway,align=center]  {\footnotesize 1-1 embedding, \\ \footnotesize Theorem \ref{theorem:pre_GENERIC_to_GENERIC}};
	\draw[double,<-] (m-2-3) -- (m-3-3) node [left,midway] {\footnotesize  generalization};
	
	\end{tikzpicture}
	\caption{This diagram illustrates the global structure of this paper; see the main text for the discussion.  Many of the terms and notation are only defined in later sections. }
	\label{fig:overview}
\end{figure}

Figure~\ref{fig:overview} illustrates these steps in a graphical manner. The bottom row is the existing theory of~\cite{MPR14}, which applies to reversible stochastic processes and gradient flows; the top two rows constitute the generalization that we construct in this paper. 

Item~\ref{innov:Ham-GENERIC-connection} above corresponds to the arrows connecting the middle and right columns. Theorem~\ref{theorem:pre_GENERIC_iff_generalized_fluctuation_symmetry} and Corollary~\ref{corollary:GENERIC_iff_generalized_fluctuation_symmetry_orthogonality} reformulate GENERIC and the new concept `pre-GENERIC' in terms of Hamiltonians $H$ and their properties,  thus generalizing Theorem~2.1 of~\cite{MPR14} to certain generalised-reversible Hamiltonians. 

Item~\ref{innov:SP-Ham-connection} appears as the corresponding arrows on the left, Proposition~\ref{proposition:pre-GENERIC_from_LDP} and Corollary~\ref{corollary:trivial_embedding_processes_pre_GENERIC_to_GENERIC}, which identify classes of stochastic processes that generate Hamiltonians with the properties as given in the middle column.

Item~\ref{innov:pre-GENERIC-to-GENERIC} corresponds to the arrows that connect the middle to the top row, Theorem \ref{theorem:pre_GENERIC_to_GENERIC} and Corollary \ref{corollary:trivial_embedding_processes_pre_GENERIC_to_GENERIC}.

These results are developed in the paper in the following way. In Section~\ref{sec:varevol} we introduce gradient flows, GENERIC systems, and pre-GENERIC systems. In Section~\ref{section:microscopic_origins}, we recall the emergence of gradient-flow structures from reversible stochastic processes as described in~\cite{MPR14}, and generalize this to certain generalised-reversible systems. Section~\ref{section:pre_GENERIC_to_GENERIC} is devoted to the step from `pre-GENERIC' to `full' GENERIC and we show there how each pre-GENERIC system can be converted without loss of generality into a full GENERIC system. 
Items~\ref{innov:Ham-GENERIC-connection}, ~\ref{innov:SP-Ham-connection} and ~\ref{innov:pre-GENERIC-to-GENERIC} are treated in Sections~\ref{section:microscopic_origins} and~\ref{section:pre_GENERIC_to_GENERIC}. 

Finally, in Section~\ref{section:final_example_AT_and_VFP} we discuss two examples, the Vlasov-Fokker-Planck equation and the Andersen thermostat \cite{An80} (item~\ref{innov:examples}). The corresponding microscopic systems are similar in that they both combine inertial effects with dissipative effects, leading to pre-GENERIC and GENERIC macroscopic systems. While the Vlasov-Fokker-Planck equation has quadratic dissipation and was treated in~\cite{DPZ13},  the Andersen thermostat is a jump process and therefore has non-quadratic dissipation. We comment in detail on the consequences of this difference. 

This paper is accompanied by the paper \cite{KrLaMaPe17} which regards the topic from a physical point of view and in which we treat a greater class of examples. 

Finally, we want to mention that the results and proofs in Sections \ref{sec:varevol}, \ref{section:microscopic_origins} and \ref{section:pre_GENERIC_to_GENERIC} are carried out for smooth finite-dimensional manifolds. We believe, however, that most results carry over to a more general context. Indeed the examples that are considered in Section \ref{section:final_example_AT_and_VFP} are based in an infinite dimensional setting.
	
	In addition, in Definition \ref{def:dissipation_potentials}, we restrict ourselves to strictly convex and continuously differentiable dissipation potentials. We believe that convexity suffices.

We have made these simplifying assumptions to achieve full rigor for the results in Sections \ref{sec:varevol}, \ref{section:microscopic_origins} and \ref{section:pre_GENERIC_to_GENERIC}, while on the other hand giving a full treatment of the various structural connections mentioned above.

\bigskip

\textbf{Acknowledgement} The authors gratefully acknowledge Manh Hong Duong, who pointed us towards the fluctuation symmetry \eqref{eqn:time_symmetry_in_GENERIC_flow_case} in an early stage of the work on this paper.

R.K. was supported by the Deutsche Forschungsgemeinschaft (DFG) via RTG 2131 High-dimensional Phenomena in Probability---Fluctuations and Discontinuity. A.L. was supported by the Interuniversity Attraction Pole - Phase VII/18 (Dynamics, Geometry and Statistical Physics) at the KU Leuven and the AFR PDR 2014-2 Grant No. 9202381 at the University of Luxembourg. M.A.P. was partially supported by NWO VICI grant 639.033.008.

\section{Variational characterizations of evolution equations}
\label{sec:varevol}

The evolution equations of this paper all are members of the class of `variational evolutions',  generalizations of the classical gradient-flow, or steepest-descent, evolution in $\bR^n$. These generalizations all have in common that a central role is played by one or more \emph{driving functionals} and one or more \emph{duality maps}. Since the role of the duality map is more easily understood in a geometric language, we use a formal geometric language throughout. 

We consider a smooth manifold $Z$, with tangent space $TZ = \bigsqcup_{z \in Z} T_z Z = \bigcup_{z \in Z} \{z\} \times T_zZ$ and cotangent space $T^*Z = \bigsqcup_{z \in Z} T_z^* Z$. We will say that $V : Z \rightarrow TZ$ is a vector-field on $Z$ if $V$ is continuous and $V(z) \in \{z\} \times T_z Z$. Without loss of generality, we will only consider the second component and write $V(z) \in T_z Z$.
	
For a function $F \in C_b(Z)$, we write $C^1(Z)$ the continuously differentiable functions, i.e. if for each element $z \in Z$ there is an appropriate derivative $\dd F(z) \in T^*_z Z$ and the map $z \mapsto \dd F(z)$ is continuous. 

\smallskip

A continuous-time evolution equation is a differential equation that describes the evolution of a quantity $z$ in a state space (or manifold) $Z$, which at its most abstract level can be written as 
\begin{equation} \label{eqn:flow}
\dot{z} = \mathbf{F}(z),
\end{equation}
where $\mathbf{F}$ is a vector field on $Z$. Every evolution equation is of the form~\eqref{eqn:flow}; \emph{variational} evolution equations are a subclass described by properties of $\mathbf F$. Two types are central to this paper: gradient flows and GENERIC systems. 

\subsection{Gradient flows}

Let $\Psi : TZ \rightarrow [0,\infty]$ and let $\Psi^* : T^*Z \rightarrow [0,\infty]$ be Legendre duals:
\begin{equation*}
\Psi^*(z,\xi)  = \sup_{v \in T_z Z} \left\{\ip{v}{\xi} - \Psi(z,v) \right\}, \qquad \Psi(z,v)  = \sup_{\xi \in T^*_z Z} \left\{\ip{v}{\xi} - \Psi^*(z,\xi) \right\}.
\end{equation*}

\begin{definition} \label{def:dissipation_potentials}
	We say that a Legendre pair $(\Psi,\Psi^*)$ are \emph{dissipation potentials} if $\Psi,\Psi^*$ are strictly convex, continuously differentiable and
	\begin{equation} \label{eqn:potentials_zero_at_zero}
	\Psi(z,0) = \Psi^*(z,0) = 0.
	\end{equation}
	We say that $\Psi$ and $\Psi^*$ are symmetric if $\Psi(z,v) = \Psi(z,-v)$ for all $(z,v) \in TZ$ and if $\Psi^*(z,\xi) = \Psi^*(z,-\xi)$ for all $(z,\xi) \in T^*Z$.
\end{definition}
Dissipation potentials satisfy $0 = \min_{v\in T_zZ} \Psi(z,v)  = \min_{\xi\in T_z^*Z} \Psi^*(z,\xi)$ (see~\cite[Sec.~1.2]{MPR14}).

\medskip
Let $S \in C^1(S)$ be some functional. By Legendre duality, we find for all $(z,v) \in TZ$ that
\begin{equation*}
\Psi(z,v) + \Psi^*\left(z, - \tfrac{1}{2}\dd S(z)\right) + \frac{1}{2} \ip{v}{\dd S(z)} \geq 0.
\end{equation*}
A gradient flow with respect to dissipation potentials $(\Psi,\Psi^*)$ is a flow for which this inequality is saturated.

\begin{definition}
	A \emph{gradient flow} with respect to a functional $S \in C^1(S)$ and dissipation potentials $(\Psi,\Psi^*)$ is a flow $\dot{z} = \mathbf{F}(z)$ such that
	\begin{equation*}
	\Psi(z,\mathbf{F}(z)) + \Psi^*\left(z, - \tfrac{1}{2}\dd S(z)\right) + \frac{1}{2} \ip{\mathbf{F}(z)}{\dd S(z)} = 0
	\end{equation*}
	for all $z \in Z$.
\end{definition}

Clearly, the functional $S$ is a Lyapunov function for the gradient flow as
\begin{equation*}
\frac{\dd}{\dd t} S(z(t)) = \ip{\mathbf{F}(z)}{\dd S(z)} = - 2 \left(\Psi(z,\mathbf{F}(z)) + \Psi^*\left(z, - \tfrac{1}{2}\dd S(z)\right)\right) \leq 0
\end{equation*}
by the non-negativity of $(\Psi,\Psi^*)$.

By the properties of Legendre duality, we find the equivalent characterizations
\begin{subequations}
	\begin{align}
	\dot z = \mathbf F(z) \text{ is an $(S,\Psi,\Psi^*)$-gradient flow}\quad
	&\Longleftrightarrow&  \dot z &= \partial_\xi \Psi^*\left(z,- \tfrac{1}{2} \dd S(z) \right)
	\label{eq:equiv-def-GF}\\
	&\Longleftrightarrow& -\frac12 \dd S(z) &= \partial_v \Psi(z,\dot z).
	\end{align}
\end{subequations}
Note that $\partial_v\Psi(z,\cdot)$ and $\partial_\xi\Psi^*(z,\cdot)$ are each other's inverse (ie., $\partial_v\Psi(z,\partial_\xi\Psi^*(z,\xi)) = \xi$ and vice versa). Also note that they are \emph{duality maps:} they map the tangent plane $T_zZ$ at~$z$ to the cotangent plane $T_z^*Z$ at $z$, in the case of $\partial_v\Psi(z,\cdot)$, and vice versa in the case of $\partial_\xi\Psi^*(z,\cdot)$. These duality maps are a central ingredient of variational evolution, both from a mathematical and from a modelling point of view. \emph{Mathematically}, they are necessary to convert the cotangent vector $\frac12 \dd S(z)$ into a tangent vector $\mathbf F(z)$ that can be equated to $\dot z$; since $\dd S(z)$ and $\dot z$ live in different spaces, the gradient-flow concept is not well defined without such a duality map. From a \emph{modelling} point of view, on the other hand, the duality maps are \emph{force-to-velocity} conversions: if $S$ is interpreted as an energy, then the derivative $\dd S$ is a generalized force; the gradient-flow equation~\eqref{eq:equiv-def-GF} describes how this force leads to the movement described by $\dot z$, through the force-to-velocity, cotangent-to-tangent map $\partial_\xi\Psi^*(z,\cdot)$

\begin{remark} \label{remark:quadratic_gradient_flow}
	An important subclass of gradient flows are those with dissipation potentials $(\Psi,\Psi^*)$ that are quadratic functions of their second arguments, i.e.\ both $v\mapsto \Psi(z,v)$ and $\xi\mapsto \Psi^*(z,\xi)$ are quadratic functions of $v$ and $\xi$ respectively. We refer to these as \emph{quadratic} gradient flows. In that case, the operators $M(z) : T^*_zZ \rightarrow T_z Z$ defined by $\xi \mapsto M(z)(\xi) := \partial_\xi \Psi^*(z,\xi )$ are \emph{linear}. 
	In addition, they satisfy the symmetry property $\ip{M(z)(\xi)}{\eta} = \ip{M(z)(\eta)}{\xi}$ for all $\eta,\xi \in T_z^* Z$; if $M$ is a matrix, this corresponds to the symmetry of $M$. Onsager termed this symmetry property the `reciprocal relations' in his seminal papers in 1931~\cite{On31}. 
\end{remark}

\begin{example}
To facilitate the exposition later on, we will be using a `running example' of a very simple, quadratic-dissipation system. To start this example, we let $Z=\bR^n$, $S:Z \to \bR$ some smooth function, and $\Psi^*(z,\xi) := \frac12 \|\xi\|_M ^2 = \frac12\xi^TM\xi$, where $M\in \bR^{n\times n}$ is a fixed, symmetric, positive semi-definite matrix.

Given these choices, the gradient-flow equation generated by $Z$, $S$, and $\Psi^*$ is the classical gradient flow equation
\begin{equation}
\label{eq:ex-GF}
\dot z = -\frac12 M \dd S(z).
\end{equation}
%
\end{example}

\subsection{GENERIC systems}
\label{subsec:GENERIC}
In addition to the well-studied gradient-flow systems, a second class of well studied systems are systems of GENERIC type (Generalized Equation for Non-Equilibrium Reversible-Irreversible Coupling~\cite{GrOt97,OtGr97,Ot05}). Such systems combine in an orthogonal way a gradient flow and a Hamiltonian flow. `Classical' GENERIC is formulated in terms of the symmetric operator $M$~\cite{GrOt97}; here we consider  a generalized version of GENERIC, where the symmetric matrix~$M$ has been replaced by the gradient of the convex functional $\Psi^*$~\cite{GrOt97,Mi11,Grmela12}.

Formally, our definition of a GENERIC equation implies an evolution equation of the form
\begin{equation}
\label{eq:GENERIC-eqn}
\dot z = L(z) \dd E(z) + \partial_\xi \Psi^*\left(z, - \tfrac{1}{2} \dd S(z)\right).
\end{equation}
The definition that we now choose is intended to allow for non-differential dissipation potentials.

\begin{definition} \label{definition:GENERIC_original_generalized}
	A GENERIC equation with respect to two functionals $S,E \in C^1(Z)$, dissipation potentials $(\Psi,\Psi^*)$ and duality maps $L$ is given by  
	\begin{equation} \label{eqn:generalizedGENERIC}
	\Psi(z,\dot{z} - L(z) \dd E(z)) + \Psi^*\left(z, - \tfrac{1}{2} \dd S(z)\right) + \frac{1}{2}\ip{\dot{z}}{\dd S(z)} = 0
	\end{equation}
	where 
	\begin{enumerate}[(a)]
		\item $(\Psi,\Psi^*)$ are symmetric dissipation potentials;
		\item  \label{item:def_GENERIC_L_antisymmetric_Jacobi} $L = \{L(z)\}_{z \in Z}$ is a family of operators $L(z) : T_z^* Z \rightarrow T_z Z$ that are `anti-symmetric': for every $z \in Z$ and $\xi,\eta  \in T_z^* Z$ we have $\ip{L(z) \xi}{\eta} = - \ip{L(z) \eta}{\xi}$, and the bracket $\{F,G \} := \ip{\dd F}{L \dd G}$, (applied to `smooth' functions $F_i : Z \rightarrow \bR$) satisfies the Jacobi identity:
		\begin{equation*}
		\{\{F_1,F_2\},F_3\} + \{\{F_3,F_1\},F_2\}  + \{\{F_2,F_3\},F_1\}  = 0.
		\end{equation*}
		\item \label{item:def_GENERIC_orthogonality} $E,S,L,\Psi^*$ satisfy the degeneracy conditions
		\begin{equation} \label{eqn:extended_orthogonality}
		L(z) \dd S(z) = 0, \qquad \Psi^*(z,\xi+\alpha \dd E(z)) = \Psi^*(z,\xi)\ \forall \alpha.
		\end{equation}
	\end{enumerate}
\end{definition}

The conditions in Definition~\ref{definition:GENERIC_original_generalized} are  generalizations of the original conditions on GENERIC as formulated by Grmela and \"Ottinger~\cite{GrmelaOttinger97,OttingerGrmela97,Oettinger05}. In brief, the operator $L$ is interpreted as defining a Poisson bracket $\{\cdot,\cdot\}$, such that the equation $\dot z = L(z) \dd E(z)$ defines a Hamiltonian system with Hamiltonian $E$. The joint equation~\eqref{eq:GENERIC-eqn} then combines this Hamiltonian dynamics with a gradient-flow dynamics determined by $S$ and $\Psi^*$. 

The degeneracy conditions impose constraints on how these two dynamics combine: they characterize a property of `non-interaction'~\cite{Mi11} between the two types of dynamics. This can be recognized for instance in the evolution of the functionals $E$ and $S$. These functionals 
are routinely interpreted as an energy and an entropy, and the properties of a GENERIC equation underwrite this: by the degeneracy conditions, we find that along the flow
\begin{align*}
\frac{\dd E}{\dd t} 
&= \ip{\dd E }{ L \dd E} +\left\langle{\dd E},{\partial_\xi \Psi^*\left(z, - \tfrac{1}{2} \dd S(z)\right)} \right\rangle\\
&= \ip{\dd E}{L \dd E} +\partial_\alpha \Psi^*\left(z, - \tfrac{1}{2} \dd S(z) + \alpha \dd E(z)\right) = 0, 
\end{align*}
whereas by the non-negativity of $(\Psi,\Psi^*)$ and the GENERIC equation
\begin{equation} \label{eqn:GENERIC_S_Lyapunov}
\frac{\dd S}{\dd t} = \ip{\dot{z}}{\dd S(z)} = - 2\left(\Psi\bigl(z,\dot{z} - L(z) \dd E(z)\bigr) + \Psi^*\left(z, - \tfrac{1}{2} \dd S(z)\right)\right) \leq 0.
\end{equation}
Therefore $E$ is conserved and $S$ decreases along a solution of~\eqref{eqn:generalizedGENERIC}. (We should emphasize  that the direction of change of $S$ is the opposite of the usual GENERIC convention: $S$ \emph{decreases} along the flow.)

Thus, a GENERIC equation characterizes a relaxation to equilibrium in terms of a dynamics that consists of two parts; one is dissipative, and is the gradient flow of an `entropy'~$S$, whereas the other is Hamiltonian, and conserves the `energy' $E$. 

\begin{example}
To continue the example, the gradient-flow equation~\eqref{eq:ex-GF} can be converted into a simple GENERIC system by adding a  conservative term
\begin{equation}
\label{eq:ex:GENERIC}
\dot z = L \dd E(z) -\frac12 M \dd S(z),
\end{equation}
where we introduced a constant antisymmetric matrix $L\in \bR^{n\times n}$ and an energy function $E \in C^1(Z)$. The second degeneracy condition in~\eqref{eqn:extended_orthogonality} reduces to $M\dd E=0$, as we see from an explicit calculation: for all $\xi\in \bR^n$, 
\[
0 = \partial_\alpha \Psi^*\bigl(z,\xi+\alpha \dd E(z)\bigr)\Big|_{\alpha=0} = \partial_\alpha \tfrac12 \|\xi+\alpha \dd E(z)\|_M^2\Big|_{\alpha=0}   = \xi^T{M\dd E(z)}.
\]
Therefore we find the two degeneracy conditions $L\dd S=0$ and $M\dd E=0$.
\end{example}

\begin{example}
We can make our `running example' even more explicit. 
A natural example of equation~\eqref{eq:ex:GENERIC} is a Hamiltonian system with friction; here $L\dd E$ represents the Hamiltonian dynamics, and $-\frac12 M\dd S$ the dissipative dynamics associated with friction. This example however can not be done with a constant matrix $M$, and therefore we will consider a position-dependent matrix $M(z)$. See~\cite[Sec.~3.1]{Mi11} for a more extensive discussion of this example.

Let $q\in \bR^d$ and $p\in \bR^d$ be canonical variables of the Hamiltonian system, with Hamiltonian $H(q,p) = \frac {p^2}{2m} + V(q) $. Upon adding friction to the system, $H$ is not conserved; the thermodynamic interpretation is that mechanical energy (represented by $H$) is converted into heat. Therefore we add an `internal energy' variable $e\in\bR$ that captures this heat; we discuss this addition in more detail in Section~\ref{section:pre_GENERIC_to_GENERIC}. The state space is then $\bR^{d+d+1}$, with variables $z=(q,p,e)$, and the total energy is $E(q,p,e) = H(q,p) + e = \frac {p^2}{2m} + V(q) +e$.

The Poisson operator $L$  simply describes the symplectic structure of the Hamiltonian system, 
\[
L = \begin{pmatrix} 0 & I &0\\-I& 0 & 0 \\ 0& 0& 0 \end{pmatrix} \in \bR^{(d+d+1)\times (d+d+1)}.
\]
The dissipative structure is given by 
\[
S(q,p,e) := -e, \qquad
M(q,p,e) := 2\gamma \begin{pmatrix}
0 & 0 & 0 \\
0 & 1 &- p/m\\
0 & - p/m & {p^2}/m^2
\end{pmatrix},
\]
where $\gamma>0$ is a friction parameter. With these choices, the evolution equation~\eqref{eq:ex:GENERIC} becomes
\begin{alignat*}3
\dot q &= L_{qp}\partial_p E &&& &=\frac pm, \\
\dot p &= L_{pq}\partial_q E  &&- \tfrac12 M_{pe} \partial_e S &&= -\partial_q V(q) - \gamma \frac pm , \\
\dot e &= &&-\tfrac12 M_{ee} \partial_e S &&= \gamma \frac{p^2}{m^2}.
\end{alignat*}
Note that the degeneracy conditions $L\dd S = 0$ and $M\dd E = 0$ are satisfied by construction. 
\end{example}

\begin{remark}
	The second statement in \eqref{eqn:extended_orthogonality} is equivalent to the notion that $\dd E$ is a \textit{distinguished} function for the bracket generated by $\Psi^*$ introduced on page 6624 of \cite{GrOt97}. 
\end{remark}

\subsection{Pre-GENERIC}

We will see below that a less constrained version of GENERIC, which we will call pre-GENERIC, emerges naturally for a broad class of evolution equations from their microscopic origins. For such equations, we do not consider a flow coming from some conserved energy~$E$, but replace $L(z) \dd E(z)$ by a general flow $W(z)$. The constraint in \eqref{item:def_GENERIC_orthogonality} of Definition \ref{definition:GENERIC_original_generalized} can then be replaced by the less constraining orthogonality condition $\ip{W(z)}{\dd S(z)} = 0$.

\begin{definition} \label{definition:pre_GENERIC}
	A pre-GENERIC equation with respect to a functional $S \in C^1(Z)$, dissipation potentials $(\Psi,\Psi^*)$ and vector field $W : Z \rightarrow TZ$ is given by  
	\begin{equation} \label{eqn:preGENERIC}
	\Psi(z,\dot{z} - W(z)) + \Psi^*\left(z, - \tfrac{1}{2} \dd S(z)\right) + \frac{1}{2}\ip{\dot{z}}{\dd S(z)} = 0
	\end{equation}
	where 
	\begin{enumerate}[(a)]
		\item $(\Psi,\Psi^*)$ are are symmetric;
		\item $S$ and $W$ satisfy the non-degeneracy condition $\ip{W(z)}{\dd S(z)} = 0$ for all $z \in Z$.
	\end{enumerate}
\end{definition}

The computation in \eqref{eqn:GENERIC_S_Lyapunov} does not change under the generalization from GENERIC to pre-GENERIC. Thus, we conclude that also in the pre-GENERIC setting the functional $S$ is decreasing under the flow.

\begin{example}
\label{ex:pre-G}
The gradient-flow equation~\eqref{eq:ex-GF} can also be converted into a pre-GENERIC system by adding a flow field $W$, 
\[
\dot z = W(z) -\frac12 M \dd S(z),
\]
which should satisfy the more limited degeneracy condition $\ip{W}{\dd S}=0$.
\end{example}

A pre-GENERIC system is GENERIC if it satisfies the following additional conditions:
\begin{enumerate}
\item $W(z) = L(z)\dd E(z)$ for some energy functional $E \in C^1(Z)$ and some family of antisymmetric operators $L$ that satisfy the Jacobi identity;
\item $L(z)\dd S(z) = 0$.
\end{enumerate}

\section{Microscopic origins of variational evolution and fluctuation symmetries} \label{section:microscopic_origins}

The discussion of the papers \cite{DPZ13,MPR14} shows that the gradient flow and GENERIC structures can be derived from the microscopic origins of the macroscopic evolution equation. We will follow a similar route, and show that the construction of dissipation potentials $(\Psi,\Psi^*)$ from the Hamiltonian or the Lagrangian of the path-space large deviations as in \cite{MPR14} for the gradient flow setting, can be extended to the setting of (generalized) GENERIC. We will recover the result of \cite{DPZ13} as a special case.

We start by introducing the path-space large deviations.

\subsection{Microscopic origins of macroscopic equations and large deviations of processes} \label{section:mo_LDP}

Consider a sequence of Markov processes $X_n$ with infinitesimal generators $(A_n, \cD(A_n))$. The macroscopic limiting equation of the processes $X_n$ can be derived via the convergence of generators. In particular, if the limiting system is deterministic, we can usually find a set of functions $\cD(A) \subseteq C_b(Z)$ and functions $f_n \in \cD(A_n)$ that converge boundedly and uniformly on compacts(buc) to $f$ and such that $A_n f_n$ converges buc to $Af$. In addition, this limiting operator is of the form $Af (z) = \ip{\mathbf{F}(z)}{\dd f(z)}$ where $\mathbf F$ is a vector field on $Z$.

Path-space large deviations give us the fluctuations around $\dot{z} = \mathbf{F}(z)$ of the type
\begin{equation} \label{eqn:LDP}
\PR[\{X_n(t)\}_{t \geq 0} \approx \{\gamma(t)\}_{t \geq 0}] \sim \exp \left\{-r_n\left( I_0(\gamma(0)) + \int_0^\infty \cL(\gamma(s),\dot{\gamma}(s)) \dd s \right) \right\}.
\end{equation}
Here $r_n$ denotes the rate of exponential decay, which could have the interpretation simply $r_n= n$, but could also be in terms of a volume depending on $n$, e.g. $r_n = n^3$. Feng and Kurtz describe an algorithm~\cite{FK06} to find the Lagrangian $\cL : TZ \rightarrow [0,\infty]$  as follows. One defines pre-limit Hamiltonians $H_n f = r_n^{-1} e^{r_nf} A_n e^{r_nf}$ with domain $\cD(H_n) = \{f \, | \, e^{r_nf} \in \cD(A_n)\}$ and shows that there is a set of functions $\cD(H)$ such that for $f \in \cD(H)$ one can find $f_n \in \cD(H_n)$ that converge buc to $f$ and such that $H_nf$ converge buc to  $Hf$. As above, we usually find that $H$ is of the form $Hf(z) = \cH(z,\dd f(z))$, where $\cH$ is a map from $T^*Z$ to $\bR$. The Lagrangian $\cL$ is then found by taking the Legendre transform of $\cH$.

The associated macroscopic equation is given by the \textit{zero-cost} flow: $\dot{z} = \bfF(z)$, where $\bfF(z) = v$ if and only if $\cL(z,v) = 0$.

\smallskip

In the setting that the processes $X_n$ are reversible with respect to measures $\pi_n$, the Lagrangian respects this reversibility. Let $S$ be the rate function of, if it is satisfied, the large deviation principle with speed $r_n$. In that case, we obtain from \eqref{eqn:LDP} with $X_n(0)$ distributed as $\pi_n$ that
\begin{equation*}
S(\gamma(0)) + \int_0^t \cL(\gamma(s),\dot{\gamma}(s)) \dd s = S(\gamma(t)) + \int_0^t \cL(\gamma(s),-\dot{\gamma}(s)) \dd s
\end{equation*}
for all $t \geq 0$ and trajectories $\gamma : [0,t] \rightarrow \bR$ with finite cost. Infinitesimally, we find the fluctuation symmetry
\begin{equation} \label{eqn:time_symmetry_in_grad_flow_case}
\cL(z,v) - \cL(z,-v) = \ip{v}{\dd S(z)}, \qquad \forall \, (z,v) \in T Z.
\end{equation}

This relation can be taken as a starting point to show that $\dot{z} = \bfF(z)$ is a gradient flow for $S$. This has been carried out in \cite{MPR14}. We will recap a part of the arguments of \cite{MPR14} and then proceed to show how GENERIC can be obtained from a generalized fluctuation symmetry.

The arguments in the next sections are purely based on the information on the mesoscopic scale, i.e. on $\cL$, $H$ and $S$. The microscopic information in the processes $X_n$ will not concern us any longer.

\subsection{Gradient flows and fluctuation symmetry}
\label{subsec:GF-and-fluct-symmetry}

In this section, we give a recap of a part of the arguments in \cite{MPR14}. In the following, we will always consider Hamiltonians $\cH : T^*Z \rightarrow \bR$ and $\cL : TZ \rightarrow [0,\infty]$ that are convex and Legendre duals of each other. Additionally, we assume $\cL \geq 0$ and that for each $z \in Z$ there is some $v \in T_z Z$ such that $\cL(z,v) = 0$. Equivalently, we assume that $\cH(z,0) = 0$ for all $z \in Z$.

\smallskip

We first turn to obtaining an equivalent characterization of the fluctuation symmetry of~$\cL$ in terms of the Hamiltonian $\cH$.

\begin{definition}
	Let $S \in C^1(Z)$ and let $\cH : T^*Z \rightarrow \bR$.
	We say that the Hamiltonian $\cH$ is \textit{reversible} with respect to $S$ if $\cH(z,\xi) = \cH(z,\dd S(z) - \xi)$ for all $(z,\xi) \in T^*Z$.
\end{definition}

\begin{proposition} \label{proposition:equiv_fluctuation_symmetry_L_reversible_H}
	Let $V$ be a co-vector field, i.e. $V(z) \in T_z^* Z$. The following are equivalent.
	\begin{enumerate}[(a)]
		\item $\cL$ satisfies the fluctuation symmetry 
		\begin{equation*} 
		\cL(z,v) - \cL(z,-v) = \ip{V(z)}{v}, \qquad \forall (z,v) \in TZ,
		\end{equation*}
		\item We have 
		\begin{equation*}
		\cH(z,\xi) = \cH(z, V(z) - \xi), \qquad \forall (z,\xi) \in T^*Z.
		\end{equation*}
	\end{enumerate}
	In particular, we find that $\cH$ is reversible with respect to $S$ if and only if $\cL$ satisfies the fluctuation symmetry with $V = \dd S$ as in \eqref{eqn:time_symmetry_in_grad_flow_case}.
\end{proposition}

\begin{proof}
	The result follows from Proposition 2.1 in \cite{MPR14}. 
\end{proof}

\begin{proposition} \label{proposition:S_is_Lyapunov_in_reversible_setting}
	Let $S \in C^1(Z)$ and let $\cH : T^*Z \rightarrow \bR$. Suppose $\cH$ is reversible with respect to $S$. Then $S$ is a Lyapunov function for the zero-cost flow.
\end{proposition}

\begin{proof}
	Let $\dot{z}(t) = \bfF(z(t))$, where $\bfF(z)$ denotes the zero-cost flow field. Then
	\begin{equation*}
	\frac{\dd}{\dd t}S(z(t)) = \ip{\dd S(z(t))}{\mathbf{F}(z(t))} = \cL(z(t),\mathbf{F}(z(t)) - \cL(z(t), - \mathbf{F}(z(t)) \leq 0.
	\end{equation*}
\end{proof}

The zero-cost flow corresponding to the a reversible Hamiltonian: $\cL(z,v) = 0 \Leftrightarrow v = \partial_\xi H(z,0)$, can be rewritten, using the symmetry property obtained from the reversibility by
\begin{equation*}
v = \partial_\xi H(z,0) = - \partial_\xi \cH(z,-\dd S(z)).
\end{equation*}
This is a first hint that $\cH$ can be used to write the flow as a gradient flow. To extend beyond this basic property and write the flow in terms of dissipation potentials $(\Psi,\Psi^*)$ that are symmetric and non-negative, the Hamiltonian needs to be appropriately shifted. Note that for a reversible Hamiltonian $\cH$ we have
\begin{equation*}
\argmin_\xi \cH(z,\xi) = \frac{1}{2} \dd S(z).
\end{equation*} 
Recentering the Hamiltonian around this minimal value, we obtain a dissipation potential $\Psi^* : T^*Z \rightarrow \bR^+$:
\begin{equation} \label{eqn:construction_Psi*_from_H}
\Psi^*(z,\xi) = \cH\left(z,\xi + \tfrac{1}{2}\dd S(z)\right) - \cH\left(z,\tfrac{1}{2} \dd S(z)\right).
\end{equation}
$\Psi$ is obtained from $\Psi^*$ by taking the Legendre transform. By construction, $\Psi$ and $\Psi^*$ are symmetric, convex and non-negative. Also, $v = \partial_\xi \cH(z,0) = \partial_\xi \Psi^*\left(z, -\frac{1}{2}\dd S(z)\right)$, which is a hallmark of the gradient flow structure. Note that \eqref{eqn:construction_Psi*_from_H} is equivalent to
\begin{equation}  \label{eqn:construction_H_from_Psi*}
\cH(z,\xi) := \Psi^*\left(z,\xi - \tfrac{1}{2} \dd S(z)\right) - \Psi^*\left(z, -\tfrac{1}{2} \dd S(z) \right).
\end{equation}

The arguments in \cite{MPR14} show that the zero-cost flow corresponding to a reversible Hamiltonian is the gradient flow for $S$ with respect to the dissipation potentials $(\Psi,\Psi^*)$ and vice-versa.

\begin{theorem}[Theorem 2.1 in \cite{MPR14}] \label{theorem:equivalence_Hreversible_gradflow}
	Let $(\cL,\cH)$ and $(\Psi,\Psi^*)$ be two pairs of Legendre duals where $\cH$ and $\Psi^*$ are functionals related as in \eqref{eqn:construction_Psi*_from_H} and \eqref{eqn:construction_H_from_Psi*}. Let $S \in C^1(Z)$ and let $\bfF$ be some vector field.
	
	Then the following two statements are equivalent:
	\begin{enumerate}[(a)]
		\item $\cH$ is a reversible Hamiltonian for $S$ and $\bfF$ is the zero-cost flow for $\cL$.
		\item $\Psi^*$ is symmetric and $\bfF$ is the gradient-flow for $S$ with respect to $(\Psi,\Psi^*)$.
	\end{enumerate}
\end{theorem}

We give the proof for completeness and later reference.

\begin{proof}
	We first prove that (a) implies (b). The zero-cost flow satisfies $\cL(z,\mathbf{F}(z)) = 0$, and for all $v$ we have $\cL(x,v) \geq 0$. Therefore
	\begin{align*}
	0 \leq \cL(x,v) & = \sup_{\xi} \left\{\ip{v}{\xi} - \cH(z,\xi) \right\} \\
	& = \sup_\xi \left\{\ip{v}{\xi} - \Psi^*(z,\xi) \right\} - \cH\left(z,\tfrac{1}{2}\dd S(z)\right) + \frac{1}{2}\ip{v}{\dd S(z)} \\
	& = \Psi(z,v) + \Psi^*\left(z, - \tfrac{1}{2} \dd S(z)\right) + \frac{1}{2} \ip{v}{\dd S(z)}.
	\end{align*}
	This implies that indeed $\dot{z} = \mathbf{F}(z) = \partial_\xi H(z,0)$ is a gradient flow for $S$ with respect to the dissipation potentials $(\Psi,\Psi^*)$.
	
	We proceed with the proof that (b) implies (a). The reversibility of $\cH$ with respect to $S$ follows from the symmetry of the functional $\Psi^*$. The gradient flow $\mathbf{F}(z)$ satisfies by \eqref{eq:equiv-def-GF}
	\begin{equation*}
	\bfF(z) = \partial_{\xi} \Psi^*\left(z, - \tfrac{1}{2} \dd S(z) \right).
	\end{equation*}
	As $\cH$ is defined by a shift of $\Psi^*$, we find $\bfF(z) = \partial_{\xi} H(z, 0)$, which is equivalent to $\cL(z,\bfF(z)) = 0$.
\end{proof}

The goal of this paper is to establish a similar connection in the setting of (pre-) GENERIC. We will establish equivalent conditions for the Lagrangian and Hamiltonian in terms of symmetries, that for the zero-cost trajectory leads to an equation that can be suitably interpreted as pre-GENERIC.

\subsection{Pre-GENERIC and fluctuation symmetry around a vector field} \label{section:Pre-GENERIC}
\label{subsec:pre-GENERIC-fluctuation-symmetry}

Corresponding to the equivalence given between the reversibility of a Hamiltonian with respect to an entropy and the time-symmetry relation \eqref{eqn:time_symmetry_in_grad_flow_case} for the associated Lagrangian, we start out with a fluctuation-symmetry that will be a first step towards GENERIC. 

Motivated by the examples considered in \cite{DPZ13}, we assume that our flow consists of a gradient part for some entropy $S \in C^1(Z)$, to which we add a vector field. The fluctuation symmetry will be as in \eqref{eqn:time_symmetry_in_grad_flow_case}, but taken around the vector field $W$:
\begin{equation} \label{eqn:time_symmetry_in_GENERIC_flow_case}
\cL(z,W(z) + v) - \cL(z,W(z) - v) = \ip{\dd S(z)}{v}, \qquad \forall (z,v) \in TZ.
\end{equation}

As has been done in \cite{MPR14} for the first fluctuation symmetry, we establish an equivalent description in terms of the Hamiltonian. Here the important observation, detailed by the proposition below, is that the act of {translating} the symmetry from the origin to $W$, as in moving from~\eqref{eqn:time_symmetry_in_grad_flow_case} to~\eqref{eqn:time_symmetry_in_GENERIC_flow_case},  is equivalent to  \emph{adding} a corresponding term to the Hamiltonian.

\begin{proposition} \label{proposition:symmetry_equivalence_L_H_intermsof_V_W}
	Consider $\cL : TZ \rightarrow [0,\infty]$ and its Legendre transform $\cH : T^*Z \rightarrow \bR$. Let $W$ be a vector field and $V$ a co-vector field, i.e. $W(z) \in T_z Z$ and $V(z) \in T^*_z Z$.	The following are equivalent:
	\begin{enumerate}[(a)]
		\item We have a decomposition $\cH = \cH_1 + \cH_2$ such that $\cH_1(z,\xi) = \ip{\xi}{W(z)}$ and $\cH_2$ is symmetric around $\tfrac{1}{2}V$: $\cH_2(z,\xi) = \cH_2(z,V(z)- \xi)$. 
		\item We have $\cL(z,W(z)+v) - \cL(z,W(z)-v) = \ip{V(z)}{v}$.
	\end{enumerate}
\end{proposition}

\begin{proof}
	Let $\cL_2$ denote the Legendre transform of $\cH_2$. Then, we have
	\begin{align*}
	\cL(z,W(z) + v) & = \sup_\xi \left\{ \ip{\xi}{W(z) + v} - \cH(z,\xi) \right\} \\
	& = \sup_\xi \left\{ \ip{\xi}{v} - \cH_2(z,\xi) \right\} \\
	& = \cL_2(z,v)
	\end{align*}
	for all $(z,v) \in TZ$. Thus, the equivalence is derived via the application of Proposition \ref{proposition:equiv_fluctuation_symmetry_L_reversible_H} for~$\cL_2$.
\end{proof}

We will use the result above for the setting where $V = \dd S$. In particular settings, e.g. for GENERIC, we will use a specific $W$ also, e.g. $W = L \dd E$.

If GENERIC is to be found via microscopic origins as in the gradient-flow setting, we expect $S$ to be a Lyapunov function for the zero-cost flow as in \eqref{eqn:GENERIC_S_Lyapunov}. A decomposition of the type above with $V = \dd S$ is not sufficient for $S$ to be a Lyapunov function for the zero-cost flow. In the argument of Proposition \ref{proposition:S_is_Lyapunov_in_reversible_setting}, we used that $\cL(z,\mathbf{F}(z)) = 0$. If we aim to use a similar trick in this setting, we need to shift our inner-product $\ip{\dd S(z)}{\dd \mathbf{F}(z)}$. Let $\dot{z}(t) = \mathbf{F}(z(t))$, then (writing $z(t) = z$ for brevity)
\begin{eqnarray}
\frac{\dd}{\dd t} S(z) & = & \ip{\dd S(z)}{\bfF(z)} \notag \\
& = & \ip{\dd S(z)}{W(z)} + \ip{\dd S(z)}{\bfF(z) - W(z)}  \notag \\
& \stackrel{\eqref{eqn:time_symmetry_in_GENERIC_flow_case}}= & \ip{\dd S(z)}{W(z)} + \cL\bigl(z,\bfF(z)\bigr) - \cL\bigl(z,2W(z) - \bfF(z)\bigr) \notag \\
& \leq & \ip{\dd S(z)}{W(z)}.
\label{eqn:pre_GENERIC_S_Lyapunov}
\end{eqnarray}
Thus, to conclude that $S$ is a Lyapunov function for the zero-cost flow, it suffices to have $\ip{\dd S(z)}{W(z)} \leq 0$ for all $z$. Since we are working towards microscopic origins of GENERIC, we will however assume that the vector field $W$ and the co-vector field $\dd S$ are orthogonal. This leads to the notion of pre-GENERIC in Definition \ref{definition:pre_GENERIC}, that does not take into account a conserved energy $E$ or operators $L$.

The following theorem generalizes Theorem \ref{theorem:equivalence_Hreversible_gradflow} to the pre-GENERIC setting, thus establishing that pre-GENERIC is the relevant notion when considering the generalized fluctuation symmetry  \eqref{eqn:time_symmetry_in_GENERIC_flow_case}.

\begin{theorem} \label{theorem:pre_GENERIC_iff_generalized_fluctuation_symmetry}
	Let $(\Psi,\Psi^*)$ be a pair of Legendre duals. Let $S \in C^1(Z)$ and let $\bfF$  be a vector field.
	
	Let $\cH_1(z,\xi) = \ip{W(z)}{\xi}$ for some vector field $W$ which is orthogonal to $\dd S$: $\ip{W(z)}{\dd S(z)} = 0$ and let $\cH_2$ be related to $\Psi^*$ as in \eqref{eqn:construction_Psi*_from_H} and \eqref{eqn:construction_H_from_Psi*}. Finally, let $\cH := \cH_1 + \cH_2$ and let $\cL$ be the Legendre dual of $\cH$.

	The following are equivalent:
	\begin{enumerate}[(a)]
		\item \label{itemThm:zero_cost} $\cH_2$ is reversible for $S$ and $\bfF$ is the zero-cost flow for $\cL$.
		\item \label{itemThm:pregeneric} $\Psi^*$ is symmetric and $\bfF$ is pre-GENERIC for $S$ with dissipation potentials $(\Psi,\Psi^*)$ and vector field $W$.
	\end{enumerate}	
	If (a) or (b) is satisfied, the functional $S$ is a Lyapunov function for $\bfF$.
\end{theorem}

As a corollary, we obtain conditions on the Hamiltonians that are equivalent to GENERIC.

\begin{corollary}\label{corollary:GENERIC_iff_generalized_fluctuation_symmetry_orthogonality}
	Consider the setting of Theorem \ref{theorem:pre_GENERIC_iff_generalized_fluctuation_symmetry}. Let $E \in C^1(Z)$ and let $L = \{L(z)\}_{z \in Z}$ be a family of operators $L(z) : T_z^* Z \rightarrow T_z Z$ that are `anti-symmetric' and satisfy the Jacobi identity (cf.\ Definition \ref{definition:GENERIC_original_generalized}\ref{item:def_GENERIC_L_antisymmetric_Jacobi}). Finally suppose that $W(z) = L(z) \dd E(z)$.

	Then the  following are equivalent:
	\begin{enumerate}[(a)]
		\item \label{itemCor:zero_cost} $\cH_2$ is reversible for $S$, $L \dd S = 0$, and the map $\alpha \mapsto \cH_2(z,\xi + \alpha \dd E(z))$ is constant; $\bfF$ is the zero-cost flow for $\cL$;
		\item \label{itemCor:generic} $\Psi^*$ is symmetric and $\bfF$ is GENERIC for $S,E$, dissipation potentials $(\Psi,\Psi^*)$ and map $L$.
	\end{enumerate}
	If (a) or (b) is satisfied, then $S$ is a Lyapunov function for $\bfF$ and the functional $E$ is constant along $\bfF$.
\end{corollary}

\begin{proof}[Proof of Theorem \ref{theorem:pre_GENERIC_iff_generalized_fluctuation_symmetry}]
	We start with the proof that \eqref{itemThm:zero_cost} implies \eqref{itemThm:pregeneric}. Let $\cL_2$ be the Lagrangian corresponding to $\cH_2$. The zero-cost flow of $\cL_2$ is given by $\bfF(z) - W(z)$. As $\cH_2$ is reversible for $S$, we find by Theorem \ref{theorem:equivalence_Hreversible_gradflow} (a) to (b) that $\bfF(z) - W(z)$ is the gradient flow for $S$ with respect to  $(\Psi,\Psi^*)$:
	\begin{equation*}
	\Psi(z,\bfF(z) - W(z)) + \Psi^*\left(z, - \tfrac{1}{2} \dd S(z) \right) + \frac{1}{2} \ip{\bfF(z) - W(z)}{\dd S(z)} = 0.
	\end{equation*}
	Using that $\ip{W(z)}{\dd S(z)} = 0$, we find that the zero-cost flow $\bfF$ of $\cL$ is pre-GENERIC for $S$, dissipation potentials $(\Psi,\Psi^*)$ and vector field $W$.
	
	\smallskip
	
	We proceed with the proof that \eqref{itemThm:pregeneric} implies \eqref{itemThm:zero_cost}. By Using that $\bfF$ is pre-GENERIC for $S$ with dissipation potentials $(\Psi,\Psi^*)$ and vector field $W$, we find by Theorem \ref{theorem:equivalence_Hreversible_gradflow} (b) to (a), and the property that $\ip{W(z)}{\dd S(z)} = 0$ that $\bfF - W$ is the zero-cost flow for $\cL_2$, the Lagrangian associated to $\cH_2$. Thus, we have $\bfF(z) - W(z) = \partial_{\xi} \cH_2(z,0)$. On the other hand 
	\begin{equation*}
	\partial_{\xi} \cH(z,0) = \partial_{\xi} \cH_1(z,0) + \partial_\xi \cH_2(z,0) = W(z) + \bfF(z) - W(z) = \bfF(z),
	\end{equation*}
	which implies that the pre-GENERIC flow is the zero-cost flow for $\cL$.
	
	\smallskip
	
	The final statement follows from \eqref{eqn:pre_GENERIC_S_Lyapunov} and the property that $\ip{W(z)}{\dd S(z)} = 0$.
\end{proof}

\begin{proof}[Proof of Corollary \ref{corollary:GENERIC_iff_generalized_fluctuation_symmetry_orthogonality}]
	Both implications are immediate.
\end{proof}

\subsection{Returning to stochastic processes: Pre-GENERIC from large deviations} \label{section:returning_to_SP}

In Section~\ref{section:mo_LDP} we explained how Hamiltonians and Lagrangians arise in the context of large-deviation principles of stochastic processes. In Sections~\ref{subsec:GF-and-fluct-symmetry} and~\ref{subsec:pre-GENERIC-fluctuation-symmetry} we temporarily disregarded the stochastic origin of these functionals, and focused instead on consequences of symmetry properties of these functionals for the deterministic evolutions that they generate. We now return to the stochastic context and consider the consequences of Sections~\ref{subsec:GF-and-fluct-symmetry} and~\ref{subsec:pre-GENERIC-fluctuation-symmetry} for stochastic processes. 

\medskip

The example that we give here is in $\bR^d$, but we believe the general structure to be much more widely valid. The infinite-dimensional examples in Section~\ref{section:final_example_AT_and_VFP} illustrate this. 

Let $Z=\bR^d$, and consider a sequence of stochastic processes $Z_n$ in $\bR^d$, with generator~$A_n$. To mirror the setup of the theorems above, we assume that $A_n$ has a decomposition $A_n = A_{1,n}+A_{2,n}$ on a common core, and that the `sub-generators' $A_{1,n}$ and $A_{2,n}$ independently generate stochastic processes $X_n$ and $Y_n$. 

In addition, we assume that 
\begin{enumerate}[(a)]
\item $A_n$, $A_{1,n}$, and $A_{2,n}$ share a common unique invariant measure $\pi_n$, and that $\pi_n$ satisfies a large deviation principle on $\bR^d$ with speed $n$ and with rate function $S\in C^1(\bR^d)$;
\item \label{assumption:pi_n-reversible}
$\pi_n$ is reversible for $A_{2,n}$;
\item Each of the processes $X_n$, $Y_n$, and $Z_n$ satisfy a large-deviation principle on $D_{\bR^d}([0,\infty))$ with speed $n$ and with rate function $I$ of Lagrangian form. The Lagrangians $\cL$, $\cL_1$, and $\cL_2$ correspond to Hamiltonians $H$, $H_1$, and $H_2$, satisfying  $H = H_1 + H_2$ and 
\[
Hf(z) = \cH(z,\dd f(z)), \qquad 
H_1f(z) = \cH_1(z,\dd f(z)), \qquad \text{and}\qquad
H_2f(z) = \cH_2(z,\dd f(z));
\] 
\item 
\label{assumption:linear_cH_1}
$\cH_{1}(z,\xi) = \ip{W(z)}{\xi}$ for some vector field $W\in C(\bR^d;\bR^d)$ and $\cH_2$ is continuous and convex in the second coordinate; 
\end{enumerate}
%
%

\begin{proposition} \label{proposition:pre-GENERIC_from_LDP}
Consider the setup above. Then the limiting dynamics are pre-GENERIC for $S$ with vector field $W$ and dissipation potentials $(\Psi,\Psi^*)$ defined in terms of $\cH_2$ as in Theorem \ref{theorem:pre_GENERIC_iff_generalized_fluctuation_symmetry}.
\end{proposition}

\begin{proof}[Proof of Proposition \ref{proposition:pre-GENERIC_from_LDP}]
We prove the proposition by showing that condition \eqref{itemThm:zero_cost} of Theorem~\ref{theorem:pre_GENERIC_iff_generalized_fluctuation_symmetry} are satisfied with $\bfF$ the limiting dynamics of the processes $Z_n$. Since $Y_n$ is a reversible stochastic process (assumption~(\ref{assumption:pi_n-reversible})), the reversibility of $\cH_2$ for $S$ is given by~\cite[Prop.~3.7]{MPR14}. 

We now show that $\cH_1(z,\dd S(z)) \stackrel{(\ref{assumption:linear_cH_1})}= \ip{W(z)}{\dd S(z)}$ vanishes. For a function $u_0 : \bR^d \rightarrow \bR$, define for $t>0$ the value function
	\begin{equation*}
	u(t,z) := \inf_{\substack{\gamma \in \cA \cC \\ \gamma(t) = z}} \left\{u_0(\gamma(0)) + \int_0^t \cL_1(\gamma(s),\dot{\gamma}(s)) \dd s \right\}.
	\end{equation*}
	By the contraction principle we find that if $u_0$ is the rate function of the  law of the process~$X_n$ at time zero, then $z \mapsto u(t,z)$ is the rate function of the law of $X_n$ at time $t$. Since $S$ is the rate function of the stationary measures, we find that if $u_0 = S$ then $u(t,\cdot) = S(\cdot)$; therefore, for all $t>0$, 
\[
S(z) = \inf_{\substack{\gamma \in \cA \cC \\ \gamma(t) = z}} \left\{S(\gamma(0)) + \int_0^t \cL_1(\gamma(s),\dot{\gamma}(s)) \dd s \right\}.
\]

Now $\cL_1(z,v)$ equals $0$ if $v=W(z)$, and $+\infty$ otherwise. Therefore this identity reduces to
\[
S(\gamma(t)) = S(\gamma(0)), 
\]
where $t>0$ is arbitrary and $\gamma:[0,t]\to\bR^d$ is any curve such that $\dot \gamma(s) = W(\gamma(s))$ for all $0\leq s\leq t$. By taking the limit $t\to0$ we find that $\ip{\dd S(z)}{W(z)} = 0$ for all $z\in \bR^d$. 

Therefore the conditions of Theorem~\ref{theorem:pre_GENERIC_iff_generalized_fluctuation_symmetry} are satisfied, and we find that the zero-cost flow is pre-GENERIC.

\end{proof}

\section{From pre-GENERIC to GENERIC} 
\label{section:pre_GENERIC_to_GENERIC}

In \cite{DPZ13} it is shown shown that the underdamped Fokker-Planck equation can be turned into GENERIC by adding an appropriate energy variable. The underdamped Fokker-Planck equation has a pre-GENERIC structure, as we will show in Section \ref{section:final_example_AT_and_VFP} below. It turns out that a pre-GENERIC system can be systematically extended into the GENERIC structure by the addition of an auxiliary variable which takes the role of the energy $E$, as we show below in Section \ref{section:adding_energy_variable}. 

This extension is not unique, and in certain contexts other extensions make more physical sense. This happens in the setting of the underdamped Fokker-Planck equation in which there is a physical energy that is not conserved under the dynamics. In Section \ref{section:change_of_variables_GENERIC}, we will consider a change of coordinates that turns our GENERIC representation of the Fokker-Planck equation in a physically more sensible one that is equivalent to the one in \cite{DPZ13}.

\subsection{Adding an energy variable for the macroscopic dynamics} \label{section:adding_energy_variable}

The structure will be as follows. Consider a pre-GENERIC system with components $S \in C^1(Z)$, $\Psi$, and $W$, and equation $\dot z = W(z) + \partial_\xi \Psi^*(z,-\tfrac12 \dd S(z))$. We now add an artificial variable $\aux\in \bR$, to form a pair $(z,\aux)\in \widehat{Z}:=Z\times \bR$. (The tangent and cotangent spaces are $T_{z,\aux}\widehat{Z} = T_z Z \times \bR$ and $T_{z,\aux}^*\widehat{Z}  = T_z Z \times \bR$.) Since a pre-GENERIC system has no concept of energy, we need to introduce one; the energy will be the mapping $\widehat E(z,\aux):= \aux$, and since the energy should be conserved the evolution equation for $\aux$ will be $\dot \aux = 0$. Theorem~\ref{theorem:pre_GENERIC_to_GENERIC} below formulates the resulting set of equations,
\begin{align*}
\dot z &= W(z) + \partial_\xi \Psi^*(z,-\tfrac12 \dd S(z)),\\
\dot \aux &= 0,
\end{align*}
as a full GENERIC system.

\begin{theorem} \label{theorem:pre_GENERIC_to_GENERIC}
	Let $S \in C^1(Z)$. Let $W$ be some vector field and let $\cH = \cH_1 + \cH_2$ with $\cH_i : T^*Z \rightarrow \bR$ and $\Psi^* : T^*Z \rightarrow \bR^+$ be functionals with $\cH_1(z,\xi) = \ip{W(z)}{\xi}$. Let $\bfF$ be a vector field.
	
	Define $\widehat{\bfF} = (\bfF,0)$, $\widehat{S}(z,\aux) = S(z)$, $\widehat{E}(z,\aux) = \aux$,
	\begin{equation*}
	\widehat{\Psi}^*\left(\begin{pmatrix}
	z \\ \aux
	\end{pmatrix}, \begin{pmatrix}
	\xi \\ r
	\end{pmatrix}\right) =  \Psi^*(z,\xi), \qquad
	\widehat{\cH}_i\left(\begin{pmatrix}
	z \\ \aux
	\end{pmatrix}, \begin{pmatrix}
	\xi \\ r
	\end{pmatrix}\right) = \cH_i(z,\xi), \quad i \in \{1,2\},
	\end{equation*}
	and
	\begin{equation*}
	\widehat{L}(z,\aux) \begin{pmatrix}
	\xi \\ r
	\end{pmatrix} = \begin{pmatrix}
	r W(z) \\ - \ip{W(z)}{\xi}
	\end{pmatrix} = \begin{pmatrix}
	0 & W(z) \\ - W(z) & 0
	\end{pmatrix} \begin{pmatrix}
	\xi \\ r
	\end{pmatrix}.
	\end{equation*}
	The maps $\widehat{\Psi}^*$ and $\widehat{\cH}_2$ are related to each other as in \eqref{eqn:construction_Psi*_from_H} and \eqref{eqn:construction_H_from_Psi*}. Furthermore the map $\widehat{L}$ is anti-symmetric and satisfies the Jacobi identity. If $\ip{W(z)}{\dd S(z)} = 0$, then $\widehat{L} \dd \widehat{S} = 0$. Finally, we have the following two results:
	\begin{enumerate}[(a)]
		\item If Theorem \ref{theorem:pre_GENERIC_iff_generalized_fluctuation_symmetry} (\ref{itemThm:zero_cost}) holds for $\cH, S$ and $\bfF$, then Corollary \ref{corollary:GENERIC_iff_generalized_fluctuation_symmetry_orthogonality} (\ref{itemCor:zero_cost}) holds for $\widehat{\cH}, \widehat{S},\widehat{E},\widehat{L}$ and $\widehat{\bfF}$.
		\item If Theorem \ref{theorem:pre_GENERIC_iff_generalized_fluctuation_symmetry} (\ref{itemThm:pregeneric}) holds for $\Psi^*, S$ and $\bfF$, then Corollary \ref{corollary:GENERIC_iff_generalized_fluctuation_symmetry_orthogonality} (\ref{itemCor:generic}) holds for $\widehat{\Psi}^*, \widehat{S},\widehat{E},\widehat{L}$ and $\widehat{\bfF}$.
	\end{enumerate}

\end{theorem}

\begin{example}
Performing this extension on the pre-GENERIC system of Example~\ref{ex:pre-G}, we find 
\[
\begin{aligned}
\widehat E(z,\aux) &:= \aux, \quad & \widehat L &:= \widehat L(z,\aux) = \begin{pmatrix}0 & W\\-W & 0\end{pmatrix},\\
\widehat S(z,\aux) &:= S(z), \qquad & \widehat M &:= \widehat M(z,\aux) = \begin{pmatrix}M & 0\\0 & 0\end{pmatrix},
\end{aligned}
\]
and therefore
\[
\partial_t \begin{pmatrix} z\\\aux \end{pmatrix}
= \begin{pmatrix}0 & W\\-W & 0\end{pmatrix} \begin{pmatrix}0\\1\end{pmatrix}
+ \begin{pmatrix}M & 0\\0 & 0\end{pmatrix}\begin{pmatrix}-\tfrac12 \dd S \\0\end{pmatrix},
\]
or 
\begin{align*}
\dot z &= W(z) -\frac12 M \dd S(z),\\
\dot \aux &= 0.
\end{align*}
The degeneracy condition $\widehat M \dd \widehat E = 0$ is satisfied by construction, and since $\widehat L\dd \widehat S = (0,-\ip W{\dd S})$, the degeneracy condition $\widehat L\dd \widehat S=0$ is satisfied by the assumption that $\ip W{\dd S} = 0$.
\end{example}

As this example illustrates, the extension with an energy variable in this way is essentially trivial: no physics is added or assumed, and the only mathematical change is a change of form. This insight raises an interesting question: Does GENERIC have any significant structure in addition to the properties of pre-GENERIC? If pre-GENERIC can be `trivially' converted into GENERIC, is there actually any difference between the two? Whatever the answer to this question, Theorem~\ref{theorem:pre_GENERIC_to_GENERIC} provides a second, independent pointer to pre-GENERIC as a relevant variational structure, in addition to the appearance of pre-GENERIC in large-deviation principles (see Proposition~\ref{proposition:pre-GENERIC_from_LDP}).

\begin{proof}[Proof of Theorem \ref{theorem:pre_GENERIC_to_GENERIC}]
	The relationship between $\widehat{\Psi}^*$ and $\widehat{\cH}_2$ is immediate. The map $\widehat{L}$ is anti-symmetric since 
	\begin{multline*}
	\left\langle{\begin{pmatrix}
		\xi_1 \\ r_1
		\end{pmatrix}}{ \widehat{L}(z,\aux) \begin{pmatrix}
		\xi_2 \\ r_2
		\end{pmatrix}}\right\rangle
	= \left\langle{\begin{pmatrix}
		\xi_1 \\ r_1
		\end{pmatrix}}{\begin{pmatrix}
		r_2 W(z) \\ -\ip{W(z)}{\xi_2}
		\end{pmatrix}}\right\rangle  \\
	= r_2 \ip{W(z)}{\xi_1} - r_1 \ip{W(z)}{\xi_2} = - \left\langle{\begin{pmatrix}
		\xi_2 \\ r_2
		\end{pmatrix}}{ \widehat{L}(z,\aux) \begin{pmatrix}
		\xi_1 \\ r_1
		\end{pmatrix}}\right\rangle.
	\end{multline*}
	The tedious proof that $\widehat{L}$ satisfies the Jacobi identity is postponed to the appendix and follows from Proposition \ref{proposition:Jacobi_pre_GENERIC_to_GENERIC}. We verify $\widehat{L} \dd \widehat{S}= 0$ by a simple calculation:
	\begin{equation*}
	\widehat{L}(z,\aux) (\dd \widehat{S}(z,\aux)) = \widehat{L}(z,\aux) \begin{pmatrix}
	\dd S(z) \\ 0
	\end{pmatrix} 
	=
	\begin{pmatrix}
	0 \\ \ip{W(z)}{\dd S(z)}
	\end{pmatrix}
	= 0.
	\end{equation*}

	We proceed with the proof of (a). First, we find
	\begin{equation*}
	\widehat{L}(z,\aux) \dd \widehat{E}(z,\aux) = \begin{pmatrix}
	1\cdot W(z) \\ - \ip{W(z)}{0}
	\end{pmatrix} = \begin{pmatrix}
	W(z) \\ 0
	\end{pmatrix},
	\end{equation*}
	establishing that $\widehat{\cH}_1$ is given by a linear paring with $\widehat{L}\dd \widehat{E}$. Clearly $\widehat{\cH}_2$ is reversible for $\widehat{S}$, due to the same property for $\cH_2$ and $S$. As $\dd \widehat{E}$ only gives a contribution in the second coordinate, and $\widehat{\cH}_2$ only depends on the first coordinates, we find that $\alpha \mapsto \widehat{\cH}_2\left((z,\aux),(\xi,r) + \alpha \dd \widehat{E}(z,\aux)\right)$ is constant.

	\smallskip
	
	Finally, we prove (b). Note that \eqref{eqn:generalizedGENERIC} is satisfied for $\widehat{S},\widehat{E}$, dissipation potentials $(\widehat{\Psi},\widehat{\Psi}^*)$ and operators $\widehat{L}$, due to the fact that $\widehat{L} \dd \widehat{E} = W$, $(\widehat{\Psi},\widehat{\Psi}^*)$ act as $(\Psi,\Psi^*)$ only on the first variable, $\dd\widehat{S}(z,\aux) = (\dd S(z),0)$ and the pre-GENERIC property of $\bfF(z)$.
	By definition of $\widehat{E}$ and $\widehat{\Psi}^*$, we find
	\begin{equation*}
	\widehat{\Psi}^*((z,\aux), (\xi,r) + \alpha \dd \widehat{E}(z,\aux)) 
	= \widehat{\Psi}^*((z,\aux), (\xi,r + \alpha)) = \Psi^*(z,\xi) ,
	\end{equation*}
	which is constant in $\alpha$.
\end{proof}

\begin{remark}
The addition of a trivial energy variable to the dynamics can also be carried out for a microscopic system, for instance in the setting of Proposition \ref{proposition:pre-GENERIC_from_LDP}, leading to the following Corollary.

\begin{corollary} \label{corollary:trivial_embedding_processes_pre_GENERIC_to_GENERIC}
	
	Consider the setup in Proposition \ref{proposition:pre-GENERIC_from_LDP} and consider the dynamics $(X_n(t),0)$ and energy $\widehat{E}(z,\aux) = \aux$. 
	
	The processes $(X_n(t),0)$ preserve $\widehat{E}$ and the limiting dynamics $(\bfF,0)$ are GENERIC with energy $\widehat{E}$, entropy $\widehat{S}$, operators $\widehat{L}$ and dissipation potentials $(\widehat{\Psi},\widehat{\Psi}^*)$ as in Theorem \ref{theorem:pre_GENERIC_to_GENERIC}.
\end{corollary}

\begin{proof}
	This result is immediate from Proposition \ref{proposition:pre-GENERIC_from_LDP} and Theorem \ref{theorem:pre_GENERIC_to_GENERIC}.
\end{proof}
\end{remark}

\subsection{Transforming GENERIC in the setting with a natural energy} \label{section:change_of_variables_GENERIC}

In Section \ref{section:final_example_AT_and_VFP} below, we will see that the underdamped Fokker-Planck equation and the Andersen thermostat are pre-GENERIC system with a vector field $W$ that can be written as $L(z) \dd E(z)$ for some physical energy $E$. In this situation the choice of Theorem~\ref{theorem:pre_GENERIC_to_GENERIC} is less natural; one would want to form a GENERIC structure in which $E(z)$ plays a role as the energy. 

However, in the examples in Section~\ref{section:final_example_AT_and_VFP} the energy $E$ is not conserved by the flow. At a microscopic level, this is the consequence of energy flowing in and out of a heat bath; this heat bath is implicitly present as the generator of the noise in the system. Therefore the natural conserved energy should be of the form $E(z) + e$, where the new variable $e$ represents the energy of the heat bath. 
In~\cite{DPZ13} it was shown that this way of extending the system indeed leads to a GENERIC structure for the extended system, and here we generalize the ideas to the nonlinear dissipative structures of this paper. 
%

\medskip
To explain the choices below we first discuss the running example. 
\begin{example}
Again starting from the pre-GENERIC example of Example~\ref{ex:pre-G}, 
we now assume that $W = L\dd E$ for some antisymmetric operator $L$ satisfying the Jacobi identity and some function $E \in C^1(Z)$. Adding a variable $e$ that describes the energy of the heat bath, as discussed above, we calculate that 
\[
0 = \partial_t (E(z) + e) = \ip{\dd E}{\dot z} + \dot e 
= \ip{\dd E}{L \dd E - \tfrac12 M\dd S} + \dot e
=  - \tfrac12\ip{\dd E}{ M\dd S} + \dot e.
\]
Therefore choosing the natural evolution $\dot e = \tfrac12\ip{\dd E}{ M\dd S}$ we are led to the system
\begin{subequations}
\label{eq:Ex-C}
\begin{align}
\dot z &= LdE(z) - \tfrac12 M S(z),\\
\dot e &= \tfrac12\ip{\dd E}{ M\dd S}.
\end{align}
\end{subequations}

This system can be obtained from the system of Theorem~\ref{theorem:pre_GENERIC_to_GENERIC} by the transformation $\Phi: (z,a) \mapsto (z,a- E(z)) =: (z,e)$, which  replaces the full energy $a$ by the heat-bath energy $a-E(z) =: e$. This structure is given by $\widetilde E(z,e) = E(z) + e$ and $\widetilde S(z,e) = S(z)$, and 
\[
\partial_t \begin{pmatrix} z\\e \end{pmatrix}
= \underbrace{\begin{pmatrix}0 & L\dd E\\-L \dd E & 0\end{pmatrix}}_{\widetilde L} \underbrace{\begin{pmatrix}\dd E\\1\end{pmatrix}}_{\dd \widetilde E}
-\frac12  \underbrace{\begin{pmatrix}M & -MdE \\ \langle-M \dd E,\,\cdot\, \rangle  & \langle \dd E,M \dd E\rangle\end{pmatrix}}_{\widetilde M}
  \underbrace{\begin{pmatrix}\dd S\\0\end{pmatrix}}_{ \dd \widetilde S}.
\]
A next natural step would be to replace the first term above using the calculation
\[
\begin{pmatrix}0 & L\dd E\\-L \dd E & 0\end{pmatrix}{\begin{pmatrix}\dd E\\1\end{pmatrix}}
= \begin{pmatrix} L\dd E\\0 \end{pmatrix}
= \begin{pmatrix} L & 0 \\ 0 & 0\end{pmatrix}
\begin{pmatrix} \dd E\\ 1\end{pmatrix},
\]
since the final product reflects more naturally the structure of the original equation for $z$. 
From the pre-GENERIC structure, however, we need not have $L\dd S = 0$, implying that a structure of this type need not satisfy the degeneracy conditions. In fact, in the examples of Section~\ref{section:final_example_AT_and_VFP} we will have $L\dd S\not=0$ but $L(\dd S - \dd E) = 0$. Note that this is compatible with $\ip{W}{\dd S} = 0$, since by the antisymmetry of $L$,
\[
\ip W{\dd S} = \ip{L \dd E}{\dd S} = \ip{L (\dd E-\dd S)}{\dd S} =0.
\]

To regain the degeneracy conditions we therefore redefine the new entropy to be $\overline S(z,e) = S(z) - E(z) - e$. The resulting system, 
\[
\partial_t \begin{pmatrix} z\\e \end{pmatrix}
= \underbrace{\begin{pmatrix}L & 0\\0  & 0\end{pmatrix}}_{\overline L} \underbrace{\begin{pmatrix}\dd E\\1\end{pmatrix}}_{\dd \overline E}
-\frac12  \underbrace{\begin{pmatrix}M & -MdE \\ \langle-M \dd E,\,\cdot\, \rangle  & \langle \dd E,M \dd E\rangle\end{pmatrix}}_{\overline M}
  \underbrace{\begin{pmatrix}\dd S-\dd E \\1\end{pmatrix}}_{ \dd \overline S},
\]
satisfies all the requirements of the GENERIC structure, and gives~\eqref{eq:Ex-C} as evolution equation.

Summarizing, we first made the choice $e = a-E(z)$ for the desired additional variable.  While there is a perfectly admissible GENERIC structure for the variable $(z,e)$, we also desired that the Poisson operator be of the form $\begin{pmatrix} L&0\\0&0\end{pmatrix}$. In the examples of Section~\ref{section:final_example_AT_and_VFP} we do not have $L\dd S = 0$, only $L(\dd S-\dd E)=0$, and to make this correspond to a degeneracy condition we translated the entropy by subtracting $E$.
\end{example} 

The next theorem describes this procedure for the nonlinear GENERIC structure of this paper, and proves the corresponding properties.

\begin{theorem} \label{theorem:pre_GENERIC_natural_variables}
	Let $S,E \in C^1(Z)$. Suppose $\bfF$ is pre-GENERIC for $S$ with dissipation potentials $(\Psi,\Psi^*)$ and vector field $W$. Suppose that $W(z) = L(z) \dd E(z)$, where
	\begin{enumerate}[(a)]
		\item $L$ are anti-symmetric operators $L(z) : T_z^*Z \rightarrow T_z Z$ that satisfy the Jacobi identity,
		\item $L(z) (\dd S(z) - \dd E(z)) = 0$.
	\end{enumerate}
	Define
	\begin{equation*}
	\overline{\Psi}^*\left(\begin{pmatrix}
	z \\ e
	\end{pmatrix}, \begin{pmatrix}
	\xi \\ r
	\end{pmatrix}\right) = \Psi^*(z,\xi -r \dd E(z)),
	\end{equation*}
	let $\overline{\Psi}$ be defined as the Legendre transform of $\overline{\Psi}^*$, set $\overline{S}(z,e) = S(z) - (E(z) + e)$, $\overline{E}(z,e) = E(z) + e$, and
	\begin{equation*}
	\overline{L}(z,e) \begin{pmatrix}
	\xi \\ r
	\end{pmatrix} = \begin{pmatrix}
	L(z) \xi \\ 0
	\end{pmatrix}.
	\end{equation*}
	Then the flow $\overline{\bfF}(z,e) = (\bfF(z),- \ip{\dd E(z)}{\bfF(z)})$ is GENERIC for $\overline{S},\overline{E}$ with dissipation potentials $(\overline{\Psi},\overline{\Psi}^*)$ and operators $\overline{L}$.
\end{theorem}

\begin{remark}
	The transition from `unbarred' to `barred' objects in Theorem~\ref{theorem:pre_GENERIC_natural_variables} can be extended to obtain a continuous scale of interpolating GENERIC systems. Let $\alpha \in [0,1]$ and define $E_\alpha(z,e) = e + E(z)$, $S_\alpha(z,e) = S(z) -\alpha (e + E(z))$, $\Psi^*_\alpha((z,e),(\xi,r)) = \Psi^*(z,\xi -r \dd E(z))$, and
	\begin{equation*}
	L_\alpha(z,e) = \begin{pmatrix}
	\alpha L(z) & (1-\alpha) W(z) \\ - (1-\alpha) W(z) & 0
	\end{pmatrix}.
	\end{equation*}
	Then $\overline{\bfF}(z,e)$ is GENERIC for $S_\alpha,E_\alpha$ with dissipation potentials $(\Psi_\alpha,\Psi^*_\alpha)$ and operators $L_\alpha$.
\end{remark}

\begin{proof}[Proof of Theorem \ref{theorem:pre_GENERIC_natural_variables}]
	
	The proof consists of two steps. We start with the GENERIC representation of Theorem \ref{theorem:pre_GENERIC_to_GENERIC}, after which we make a coordinate change in the energy variable. This gives us GENERIC in a new set of variables for a changed energy and a changed dissipation potential. Then we make a second step, simultaneously changing the entropy functional and the anti-symmetric operator.
	
	\bigskip
	
	\textit{Step 1.} Denote $\widetilde{Z} = Z \times \bR$. Let $\Phi : \widehat{Z} \rightarrow \widetilde{Z}$ be defined as $\Phi(z,e) = (z,e - E(z))$. Then $D\Phi : T\widehat{Z} \rightarrow T\widetilde{Z}$ and $(D\Phi)^* : T^* \widetilde{Z} \rightarrow T^* \widehat{Z}$ and
	\begin{equation}
	\begin{aligned}
	D\Phi \left((z,e), \begin{pmatrix}
	v \\ v_e
	\end{pmatrix}\right) & = \left((z,e-E(z)), \begin{pmatrix}
	v  \\ v_e - \ip{\dd E(z)}{v} \end{pmatrix}\right), \\
	(D\Phi)^* \left((z,e), \begin{pmatrix}
	\xi \\ r
	\end{pmatrix} \right) & = \left((z,e+E(z)), \begin{pmatrix}
	\xi - \dd E(z) \\  r
	\end{pmatrix}\right).
	\end{aligned}
	\end{equation}
	
	We obtain dynamics on $\widetilde{Z}$ by pushing forward the GENERIC dynamics on $\widehat{Z}$ via the map~$\Phi$. Thus, all operators change under the action of $\Phi$, and we obtain that the flow
	\begin{equation} \label{eqn:extend_energy_flow}
	\partial_t \begin{pmatrix}
	z(t) \\ e(t)
	\end{pmatrix} = \begin{pmatrix}\mathbf{F}(z(t)) \\
	- \ip{\dd E(z(t)}{\mathbf{F}(z(t))}
	\end{pmatrix}
	\end{equation}
	is a GENERIC system in new coordinates with functionals
	\begin{align*}
	\widetilde{E}(z,e) & = \widehat{E} \circ \Phi^{-1}(z,e) = e + E(z), \\
	\widetilde{S}(z,e) & = \widehat{S}\circ \Phi^{-1}(z,e) = S(z), \\
	\widetilde{L}\left((z,e), \begin{pmatrix}
	\xi \\ r
	\end{pmatrix} \right) & = D \Phi \circ \widehat{L} \circ (D\Phi)^*\left((z,e), \begin{pmatrix}
	\xi \\ r
	\end{pmatrix} \right) = \left((z,e), \begin{pmatrix}
	0 & W(z) \\ - W(z) & 0
	\end{pmatrix} \begin{pmatrix}
	\xi \\ r
	\end{pmatrix}\right)
	\end{align*}
	and
	\begin{multline*}
	\widetilde{\Psi}^*\left((z,e), \begin{pmatrix}
	\xi \\ r
	\end{pmatrix}\right) = \widehat{\Psi}^*\left((D\Phi)^*\left((z,e), \begin{pmatrix}
	\xi \\ r
	\end{pmatrix} \right)\right) \\
	= \widehat{\Psi}^*\left((z,e+E(z)),\begin{pmatrix}
	\xi - r \dd E(z) \\ r
	\end{pmatrix}\right) = \Psi^*(z,\xi -r \dd E(z)).
	\end{multline*}
	
	\textit{Step 2.}
	Set $\overline{E} = \widetilde{E}$, $\overline{\Psi}^* := \widetilde{\Psi}^*$ and $\overline{\Psi} = \widetilde{\Psi}$. Next, we change both the entropy and the operator $L$ to obtain the desired GENERIC representation. Set $\overline{S} = \widetilde{S} - \widetilde{E}$ and 
	\begin{equation*}
	\overline{L}=  \begin{pmatrix}
	L &  0 \\ 0  & 0
	\end{pmatrix}.
	\end{equation*}
	We prove that \eqref{eqn:extend_energy_flow} is also GENERIC for $(\overline{S},\overline{E}, \overline{\Psi},\overline{\Psi}^*)$.
	
	First of all, note that $\overline{L}$ is anti-symmetric and satisfies the Jacobi identity because $L$ does. Thus, we need to show that the equation in \eqref{eqn:generalizedGENERIC} for the GENERIC representation in step~1, implies the same equation for our new representation. In addition, we need to check the non-degeneracy condition, \eqref{eqn:extended_orthogonality}. We start with the first part of the latter. We have
	\begin{equation*}
	\overline{L} \dd \overline{S} = \overline{L} (\dd \widetilde{S} - \dd \widetilde{E}) = \begin{pmatrix}
	L &  0 \\0  & 0
	\end{pmatrix} \begin{pmatrix}
	\dd S(z) - \dd E(z) \\ 1
	\end{pmatrix} = 0
	\end{equation*}
	by assumption (b). For the second part of \eqref{eqn:extended_orthogonality}note that $\widetilde{\Psi}^* = \overline{\Psi}^*$ is constant by addition of $\widetilde{E}$ in the momentum variable by the GENERIC property of the system obtained in step 1.
	
	\smallskip
	
	We proceed with the verification of \eqref{eqn:generalizedGENERIC} for our new representation by comparing the contributions of the old and new representation. In particular, we prove that
	\begin{align}
	\overline{\Psi}(z,\dot{z} - \overline{L}(z) \dd \overline{E}(z)) & = \widetilde{\Psi}(z,\dot{z} - \widetilde{L}(z) \dd \widetilde{E}(z)), \label{eqn:GENERIC_same_1} \\
	\overline{\Psi}^*\left(z, - \tfrac{1}{2} \dd \overline{S}(z)\right) & = \widetilde{\Psi}^*\left(z, - \tfrac{1}{2} \dd \widetilde{S}(z)\right), \label{eqn:GENERIC_same_2} \\
	\frac{1}{2}\ip{\dot{z}}{\dd \overline{S}(z)} & = \frac{1}{2}\ip{\dot{z}}{\dd \widetilde{S}(z)}. \label{eqn:GENERIC_same_3}
	\end{align}
	For \eqref{eqn:GENERIC_same_1} note that $\overline{\Psi} = \widetilde{\Psi}$, $\widetilde{E} = \overline{E}$ and $\overline{L} \dd \overline{E} = \widetilde{L} \dd \widetilde{E}$.  \eqref{eqn:GENERIC_same_2} follows because $\overline{\Psi}^* = \widetilde{\Psi}^*$ and the property of $\alpha \mapsto \widetilde{\Psi}^*((z,e), (\xi,r) + \alpha \dd \widetilde{E}(z,e))$ being constant. The final equality \eqref{eqn:GENERIC_same_3} follows from the fact that $\widetilde{E}$ is constant under the flow by the GENERIC property for $(\widetilde{S},\widetilde{E},\widetilde{\Psi},\widetilde{\Psi}^*)$.

	%
	%
	%
	%
	%
	%
	%
	%
	%
	
	%
	%
\end{proof}

\section{Two related examples: the Andersen thermostat and the Vlasov-Fokker-Planck equation} \label{section:final_example_AT_and_VFP}

\subsection{The macroscopic dynamics  and pre-GENERIC}

In this final section, we work out the main idea's in this paper for the the Andersen thermostat and the Vlasov-Fokker-Planck equation. In both cases, the evolution equation that we are considering is given by  $\dd_t \rho = \mathbf{F}(\rho)$ with
\begin{equation} \label{eqn:underdamped_dynamics_general}
\mathbf{F}(\rho) = - \dv_q\left(\rho \frac{p}{m}\right) + \dv_p\left( \rho \nabla_q V + \rho (\nabla_q \Phi \ast \rho)\right) + \cG^*(\rho).
\end{equation}
Here $\rho \in \cP(\bR^{2d})$ where $\bR^{2d}$ represents the space of pairs $(q_1,\dots,q_d,p_1,\dots,p_d)$ of position and momentum in $\bR^d$. We implicitly think of this vector field being defined on the tangent bundle $T \cP(\bR^{2d})$ with $T_\rho \cP(\bR^{2d}) = H^{-1}(\rho)$. The subscripts $p,q$ in $\dv_p,\dv_q,\Delta_p$ mean that the operators act only on the corresponding variable. $V,\Phi : \bR^d \rightarrow \bR^d$ are two potentials that grow sufficiently fast at infinity and we assume that $\Phi$ is symmetric. The convolution is given by
\begin{equation*}
\left(\Phi \ast \rho\right)(q) := \int \Phi(q - q') \rho(\dd q', \dd p')
\end{equation*}
and is a function of $q$ only. The dynamics $\cG^*$ is given by the adjoint of an operator $\cG$ that can take various forms. Two main examples are given by:
\begin{description}
	\item[Andersen Thermostat] redistribution of the momentum via a jump-processes with a Maxwellian distribution
	\begin{equation}
	\label{def:G-Andersen}
	\cG f(q,p) := \gamma \int \left[f(q,p') - f(q,p) \right] \frac{1}{\sqrt{2 \pi m}} e^{- \frac{p^2}{2 m}} \, \dd p,
	\end{equation}
	\item[Vlasov-Fokker-Planck] a diffusion process modeling stochastic forcing with friction
	\begin{equation}
	\label{def:G-VFP}
	\cG f(q,p) := - \gamma \left\langle{\frac{p}{m}}{ \nabla_p f(q,p)}\right\rangle  + \gamma \Delta_p f(q,p).
	\end{equation}
\end{description}


In both cases $m > 0$ models the mass of the particles, whereas $\gamma$ models the intensity of the interactions with a background solvent. In general, we can consider more general $\cG$, e.g. certain types of L\'{e}vy processes. We first work out these two examples and come back to general conditions under which the analysis can be carried out in Condition \ref{condition:reversible_wrt_gaussian} below.

At this point, it is important to note that the operator $\cG$ acts on the momentum variable only and is independent of the value of $q$, i.e. if $f$ depends on $p$ only, then $\cG f(q,p)$ depends on~$p$ only. This assumption is natural from a physical point of view as $\cG$ models the interaction between our particles and a background solvent.

In addition, the operator $\cG$ interpreted as acting on $p$ only is reversible with respect to a centered Gaussian measure
\begin{equation*}
\rho_m(\dd p) = \frac{1}{\sqrt{2\pi m}} e^{-\frac{p^2}{2m}}\dd p,
\end{equation*}
with $m > 0$. 

\smallskip

We introduce some additional structure to study the macroscopic equations in relation to pre-GENERIC and GENERIC. First of all, define $S,E : \cP(\bR^{2d}) \rightarrow \bR$ by
\begin{align}
E(\rho) & := \int \frac{p^2}{2m} + V(q) + \frac{1}{2} (\Phi \ast \rho)(q) \, \rho(\dd q, \dd p), \label{eqn:example_enegy}\\
S(\rho) & := \int \log \rho(q,p) \, \rho(\dd q, \dd p) + E(\rho). \label{eqn:example_entropy}
\end{align}
which have functional derivatives
\begin{align}
\dd E(\rho) (q,p) & = \frac{p^2}{2m} + V(q) + (\Phi \ast \rho)(q), \\
\dd S(\rho) (q,p) & =  \log \rho(q,p) + 1 + \dd E(\rho).
\end{align}
Next, let $J$ be the $2d$-dimensional symplectic matrix
\begin{equation*}
J = \begin{bmatrix}
0 & - \bONE_d \\
\bONE_d & 0
\end{bmatrix},
\end{equation*}
and define $L : T^*\cP(\bR^{2d}) \rightarrow T\cP(\bR^{2d})$ by
\begin{equation*}
L(\rho) \xi := \dv\left(\rho J \nabla \xi \right).
\end{equation*}
Note that $L$ inherits the symplectic structure of $J$, and is therefore a family of anti-symmetric operators satisfying the Jacobi identity.

\begin{theorem} \label{theorem:AT_and_VFP_are_preGENERIC}
	The Andersen thermostat and the Vasov-Fokker-Planck equation are pre-GENERIC with entropy $S$, vector field $W(\rho) := L(\rho) \dd E(\rho)$ and dissipation potentials
	\begin{align*}
	\Psi^*(\rho,\xi) & := \gamma \int \int \left[ \cosh\left( \xi(q,p') -  \xi(q,p)\right) - 1 \right] \frac{1}{\sqrt{2 \pi m}} e^{- \frac{(p^2 + (p')^2}{4 m}} \sqrt{\rho(q,p')\rho(q,p)} \, \dd p' \,\dd p \, \dd q, \\
	\Psi^*(\rho,\xi) & := \gamma \vn{\xi}_{H^1_p(\rho)}^2 = \gamma \int \left| \nabla_p \xi(q,p) \right|^2 \dd \rho(q,p),
	\end{align*}
	respectively.
\end{theorem}

Before giving the proof of the result, we start with some preliminary properties of the drift and the operators $L$ that are needed at various points in this section.

\begin{lemma} \label{lemma:example_representation_of_drift_and_orthogonality}
	We have
	\begin{enumerate}[(a)]
		\item
		\begin{equation*}
		W(\rho) := L(\rho)\dd E(\rho) = - \dv_q\left(\rho \frac{p}{m}\right) + \dv_p\left( \rho \nabla_q V + \rho (\nabla_q \Phi \ast \rho)\right),
		\end{equation*}
		\item 
			\begin{equation*}
		L(\rho) (\dd S(\rho) - \dd E(\rho)) = 0.
		\end{equation*}
	\end{enumerate}
\end{lemma}

\begin{proof}
	Both claims follow by direct calculation:
	\begin{equation*}
	L(\rho)\dd E(\rho)  = \dv J \nabla \dd E(\rho) = \dv_p\left(\rho \nabla_q V(q) + (\Phi \ast \rho)(q) \right) - \dv_q \left(\rho \frac{p}{m} \right),
	\end{equation*}
	and
	\begin{multline*}
	L(\rho)(\dd S(\rho) - \dd E(\rho))  = \dv J \nabla \left( \log \rho \right)  = \dv_p\left(\rho \nabla_q \log \rho\right) - \dv_q\left(\rho \nabla_p \log \rho\right) \\
	= \dv_p\left(\nabla_q\rho\right) - \dv_q\left( \nabla_p \rho\right) = 0.
	\end{multline*}
\end{proof}

\begin{proof}[Proof of Theorem \ref{theorem:AT_and_VFP_are_preGENERIC}]
	We start by checking property (a) and (b) of Definition \ref{definition:pre_GENERIC}. Clearly, for both the Andersen thermostat and the Vlasov-Fokker-Planck equation the potential $\Psi^*$ is non-negative, symmetric and $\Psi^*(\rho,0) = 0$. By the properties of the Legendre transform the same follows for $\Psi$. Property (b) follows by using the anti-symmetry of $L$ twice and then Lemma \ref{lemma:example_representation_of_drift_and_orthogonality} (b) (we drop dependence on $\rho$):
		\begin{equation*}
		\ip{W}{\dd S} = \ip{L \dd E}{\dd S} = \ip{L \dd E}{\dd S - \dd E} = - \ip{\dd E}{L(\dd S - \dd E)} = 0.
		\end{equation*}
	Thus, we need to check \eqref{eqn:preGENERIC}. By the properties of the Legendre transform, it suffices to establish $\bfF(\rho) = \partial_{\xi} \Psi^*\left(\rho, - \frac{1}{2} \dd S(\rho)\right)+ W(\rho)$, or by Lemma \ref{lemma:example_representation_of_drift_and_orthogonality} (a), equivalently that $\partial_{\xi} \Psi^*\left(\rho, - \frac{1}{2} \dd S(\rho)\right) = \cG^*(\rho)$.

	We start with the calculation of $\partial_{\xi} \Psi^*(\rho,\xi)$, which is given by the property that for any $\eta$ in the cotangent space, we have
	\begin{equation*}
	\ip{\eta}{\partial_{\xi} \Psi^*(\rho,\xi)} = \lim_{t \rightarrow 0} \frac{\Psi^*(\rho,\xi + t \eta) - \Psi^*(\rho,\xi )}{t}.
	\end{equation*}
	For the Andersen Thermostat, we have
	\begin{multline*} 
	\lim_{t \rightarrow 0} \frac{\Psi^*(\rho,\xi + t \eta) - \Psi^*(\rho,\xi)}{t} = \\
	\gamma \int \int \left[\eta(q,p') - \eta(q,p)\right] \sinh\left(\xi(q,p') -  \xi(q,p)\right)  \frac{1}{\sqrt{2 \pi m}} e^{- \frac{(p')^2 + p^2}{4 m}} \sqrt{\rho(q,p')\rho(q,p)} \, \dd p' \, \dd p \, \dd q.
	\end{multline*}
	
	Expanding the hyperbolic sine, we find
	\begin{align*}
	& \left\langle \partial_{\xi} \Psi^*\left(\rho,-\tfrac{1}{2}\dd S(\rho)\right),\eta\right \rangle \\
	& \qquad = \frac{\gamma}{2} \int \int \left[\eta(q,p') - \eta(q,p)\right]  \frac{1}{\sqrt{2 \pi m}} e^{- \frac{(p')^2}{2 m}} \rho(q,p') \, \dd p' \, \dd p \, \dd q \\
	& \qquad \qquad - \frac{\gamma}{2} \int \int \left[\eta(q,p') - \eta(q,p)\right]  \frac{1}{\sqrt{2 \pi m}} e^{- \frac{p^2}{2 m}} \rho(q,p) \, \dd p' \, \dd p \, \dd q \\
	& \qquad = \gamma \int \int \left[\eta(q,p') - \eta(q,p)\right]  \frac{1}{\sqrt{2 \pi m}} e^{- \frac{(p')^2}{2 m}} \rho(q,p') \, \dd p' \, \dd p \, \dd q \\
	& \qquad = \ip{\eta}{\cG^*(\rho)}.
	\end{align*}

	For the Vlasov-Fokker-Planck case, we have
	\begin{equation*}
	\Psi^*(\rho,\xi + t \eta) - \Psi^*(\rho,\xi) = 2t \gamma \int \int \nabla_p \xi \cdot \nabla_p \eta \dd \rho + t^2 \gamma \int \int \nabla_p \eta \cdot \nabla_p \eta \dd \rho.
	\end{equation*}
	Dividing by $t$ and sending $t$ to zero, we find
	\begin{equation*}
	\ip{\eta}{\partial_{\xi} \Psi^*(\rho,\xi)} = 2\gamma \int \int \nabla_p \xi \cdot \nabla_p \eta \dd \rho = \ip{\eta}{- 2 \dv_p \left(\rho \nabla_p \xi\right)}
	\end{equation*}
	establishing that $\partial_{\xi} \Psi^*(\rho,\xi) = - 2 \dv_p \left(\rho \nabla_p \xi\right)$. Filling out $- \frac{1}{2} \dd S(\rho)$, we obtain
	\begin{align}
	\partial_{\xi} \Psi^*\left(\rho,-\tfrac{1}{2}\dd S(\rho)\right) = - 2\gamma \dv_p \left(\rho \nabla_p \tfrac{1}{2}\dd S(\rho) \right)  & =  \gamma \dv_p \left(\rho \nabla_p \left(\log \rho(q,p) + 1 + \frac{p^2}{2m} \right) \right) \notag \\
	& = \gamma \dv_p \left(\nabla_p \rho(q,p) + \rho(q,p) \frac{p}{m} \right) \notag \\
	& = \gamma \Delta_p \rho(q,p) + \gamma \dv_p \left(\rho(q,p) \frac{p}{m} \right) \notag \\
	& = \cG^*(\rho). \label{eqn:VFP_cG_divEntropy}
	\end{align}
\end{proof}

We will show that the dynamics \eqref{eqn:underdamped_dynamics_general} can be obtained from a microscopic model of interacting particles which via their large deviations give rise to a pre-GENERIC structure that can extended in a natural way to GENERIC. 

\smallskip

It should be noted that we will not consider the difficulties involved in proving path-space large deviation principles for the models considered.

\subsection{Microscopic model and its description in terms of macroscopic variables}

We consider $n$ interacting particles with position and momentum $(Q_i,P_i) \in \bR^d \times \bR^d$. The dynamics of our particles consist of three separate parts:
\begin{enumerate}[(a)]
	\item independent fluctuations of the momentum variables, modeling the interaction with some background solvent;
	\item a conservative part that describes the interplay between position and momentum;
	\item pair interactions between the particles that influence the momentum.
\end{enumerate}

\textit{Part (a).} The fluctuation of the momentum variable is described by a generator $\cG \subseteq C_b(\bR^d) \times C_b(\bR^d)$, given by~\eqref{def:G-Andersen} for the Andersen Thermostat and~\eqref{def:G-VFP} for the Vlasov-Fokker-Planck case.

\textit{Part (b).} For the interplay between position and momentum, we consider dynamics for which each particle, independent of every other particle, is given by the generator
\begin{equation*}
\cB f(q,p) := \left\langle{\frac{p}{m}},{\nabla_q f(q,p)}\right\rangle - \ip{\nabla_q V(q)}{\nabla_p f(q,p)}
\end{equation*}
modeling the speed proportional to the momentum, and the change in momentum under the influence of the potential $V$. 

\textit{Part (c).} Finally, we consider the interaction amongst particles. The dynamics of the $i$-th particle depends on the locations of the other particles, and this dependence  is represented by the  generator
\begin{equation*}
\cC^i f(q_1,p_1,\dots,q_n, p_n) := - \frac{1}{2n} \sum_{j=1}^n \nabla \Phi(q_i - q_j) \nabla_{p_i} f(q_1,p_1,\dots,q_n, p_n).
\end{equation*}

Combining these operators, writing $y_i = (q_i,p_i)$, the total dynamics $\{Y_i(t)\}_{1 \leq i \leq n, t \geq 0} = \{(Q_i(t),P_i(t))\}_{1 \leq i \leq n, t \geq 0}$ has generator
\begin{equation} \label{eqn:example_microscopic_dynamics}
\begin{aligned}
\cA_n f(y_1,\dots,y_n) & = \sum_{i =1}^n \left(\cG f(y_1,\dots, y_{i-1},\,\cdot\,,y_{i+1},\dots,y_n)\right)(y_i) \\
& \qquad + \sum_{i =1}^n \left(\cB f(y_1,\dots, y_{i-1},\,\cdot\,,y_{i+1},\dots,y_n)\right)(y_i) \\
& \qquad + \sum_{i =1}^n \cC^i f(y_1,\dots,y_n).
\end{aligned}
\end{equation}

%
%
%
%
%


Our goal is to show that the dynamics in \eqref{eqn:example_microscopic_dynamics}, after taking empirical densities, converges to the solution of \eqref{eqn:underdamped_dynamics_general}.

Let $\eta_n : (\bR^{2d})^n \rightarrow \cP(\bR^{2d})$ be the map
\begin{equation*}
\eta_n(y_1,\dots, y_n) := \frac{1}{n} \sum_{j=1}^n \delta_{y_i}
\end{equation*}
for the empirical measure of $n$ particles. Since the action of the operators $\cG$ and $\cB_n$ on the $i$-th particle only depends on the particle itself, and the interaction of the $i$-th particle with other particle dynamics as induced by $\cC^i$ can be written in terms of the empirical density,
\begin{equation*}
\cC^i f(y_1,\dots,y_n) := -  \ip{(\nabla \Phi \ast_q \eta_n(y_1,\dots,y_n))(q_i)}{\nabla_{p_i} f(y_1,\dots,y_n)},
\end{equation*}
the dynamics of $\{\eta_n(Y_q(t), \dots, Y_n(t))\}_{t \geq 0}$ is autonomous. For a suitable class of functions $F$ on $\cP(\bR^{2d})$ the generator $A_n$ of these dynamics satisfies
\begin{equation} \label{eqn:example_change_of_variables_formula}
A_n F(\eta_n(y_1,\dots,y_n)) = \left[\cA_n (F \circ \eta_n) \right](y_1,\dots,y_n).
\end{equation}

To obtain an explicit representation for $A_n$, we first introduce a convenient class of test functions on $\cP(\bR^{2d})$. We say that $F$ is a cylinder function if there are $f_1,\dots,f_l \in C^2_b(\bR^{2d})$ and $\phi \in C^2(\bR^l)$ so that
\begin{equation*}
F(\rho) = \phi(\ip{f_1}{\rho},\dots,\ip{f_l}{\rho}) =: \phi(\ip{\mathbf{f}}{\rho})
\end{equation*} 
Here $\mathbf{f}$ denotes the vector $(f_1,\dots,f_l)$ and $\ip{\mathbf{f}}{\rho}$ denotes the vector of integrals $(\ip{f_1}{\rho},\dots,\ip{f_l}{\rho})$. Note that the functional derivative of the $F$ is given by 
\begin{equation*}
\dd F(\rho) = \sum_{j=1}^l \partial_j \phi(\ip{\mathbf{f}}{\rho}) f_i.
\end{equation*}

To identify the operators of the Markov process $\eta_n(Y_1(t),\dots,Y_n(t))$, we first calculate the contributions $G_n$, $B_n$, and $C_n$ of parts (a)-(c) in terms of the empirical measure. The generator~$\cG_n$ corresponding of the action of $\cG$ on $n$ particles with coordinates $y = (q,p)$ is given by
\begin{equation*}
\cG_n f(y_1,\dots,y_n) = \sum_{i =1}^n \left( \cG f(y_1,\dots,y_{i-1},\cdot,y_{i+1},\dots,y_n)\right)(y_i),
\end{equation*}
after which an application of \eqref{eqn:example_change_of_variables_formula},
\begin{equation} \label{eqn:definition_Gn}
G_n F(\eta_n(y_1,\dots,y_n)) = \left[\cG_n (F \circ \eta_n) \right](y_1,\dots,y_n),
\end{equation}
yields for the Andersen thermostat 
\begin{equation} \label{eqn:example_generator_macro_anderson}
\begin{aligned}
G_n F(\rho)  & = \frac{\gamma}{n} \sum_{i = 1}^n \int \left[\frac{F\left(\rho - \frac{1}{n} \delta_{(q_i,p_i)} + \frac{1}{n} \delta_{(q_i,p')} \right) - F(\rho)}{n^{-1}} \right] \frac{1}{\sqrt{2 \pi m}} e^{- \frac{p^2}{2 m}} \, \dd p' \\
& = \gamma \int \int \left[\frac{F\left(\rho - \frac{1}{n} \delta_{(q,p)} + \frac{1}{n} \delta_{(q,p')} \right) - F(\rho)}{n^{-1}} \right] \frac{1}{\sqrt{2 \pi m}} e^{- \frac{p^2}{2 m}} \, \dd p' \, \rho(\dd (q,p)),
\end{aligned}
\end{equation}
and for the Vlasov-Fokker-Planck equation
\begin{equation}\label{eqn:example_generator_macro_VFP}
\begin{aligned}
G_n F(\rho) & = \sum_{i = 1}^l \partial_i \phi(\ip{\mathbf{f}}{\rho}) \ip{\cG f_i}{\rho} +  \frac{\gamma}{n} \sum_{i,j = 1}^l \partial_i \partial_j \phi(\ip{\mathbf{f}}{\rho}) \ip{\nabla_p f_i \nabla_p f_j}{\rho}, \\
& = \ip{\dd F(\rho)}{\cG^*(\rho)} +  \frac{\gamma}{n} \sum_{i,j = 1}^l \partial_i \partial_j \phi(\ip{\mathbf{f}}{\rho}) \ip{\nabla_p f_i \nabla_p f_j}{\rho}.
\end{aligned} 
\end{equation}
Taking the limit in both cases, we find $\lim_n G_n F(\rho) = \ip{\dd F(\rho)}{\cG^*(\rho)}$. The operators $B_n,C_n$ given by
\begin{align*}
B_n F(\rho) & = \ip{\cB(\dd F(\rho))}{\rho}, \\
C_n F(\rho) & = - \int \ip{ \nabla_q \Phi \ast_q \rho}{\nabla_p \dd F(\rho)} \dd \rho,
\end{align*} 
are independent of $n$ and thus have limits $\ip{\dd F(\rho)}{\cB^*(\rho)}$ and $\ip{\dd F(\rho)}{\cC^*(\rho)}$ respectively, where $\cC^*(\rho) := \dv_p \left(\rho \nabla_q\left(\Phi \ast_q \rho\right)\right)$.
Thus, the dynamics on the level of the empirical measure is given by the generators $A_n = G_n + B_n + C_n$. In the limit, we find a first order generator
\begin{equation*}
\lim_n A_n F(\rho) = \ip{\dd F(\rho)}{\bfF(\rho)} = \ip{\dd F(\rho)}{\cB^*(\rho) + \cC^*(\rho) + \cG^*(\rho)}
\end{equation*}
corresponding to the dual formulation of \eqref{eqn:underdamped_dynamics_general}.

\subsection{Stationary and path-space large deviations}

A straightforward calculation using integration by parts yields that the measure
\begin{equation} \label{eqn:example_stationary_measures}
\mu_n(\dd q_1, \dd p_1, \dots, \dd q_n, \dd p_n) = Z^{-1}_n e^{-\left(\sum_{i =1}^n V(q_i)  + \frac{1}{2m}p_i^2 + \frac{1}{2n} \sum_{i \neq j} \Phi(q_i - q_j) \right)} \prod_{i =1}^n \dd q_i \, \dd p_i 
\end{equation}
is stationary for the dynamics generated by $\cA_n$ as in \eqref{eqn:example_microscopic_dynamics}.

\begin{lemma}
	Consider the stationary measures $\mu_n$ of \eqref{eqn:example_stationary_measures}. The measures $\mu_n \circ \eta_n^{-1}$ satisfy the large deviation principle on $\cP(\bR^{2d})$ at speed $n$ with rate function $S$ as in \eqref{eqn:example_entropy}.
\end{lemma}

\begin{proof}
	The measure in \eqref{eqn:example_stationary_measures} can be rewritten as
	\begin{equation*}
	\mu_n(\dd y_1, \dots, \dd y_n) = Z^{-1}_n e^{-n \left(\int V(q) + \frac{1}{2m}p^2 + \frac{1}{2} \left[\Phi \ast (\eta_n(y_1,\dots,y_n))\right](y_1,\dots,y_n) \dd \eta_n(y_1,\dots,y_n) \right)} \prod_{i =1}^n \dd y_i.
	\end{equation*}
	Thus, taking out the dependent parts involving $\Phi$, applying Sanov's theorem, and then using Varadhan's Lemma to re-insert the dependent parts, the result follows.
\end{proof}

For the path-space large deviations, we follow the approach as in Section \ref{section:mo_LDP}. First, we calculate
\begin{equation*}
H_n F(\rho) = \frac{1}{n} e^{-nF(\rho)} \left(A_n e^{nF}\right)(\rho)
\end{equation*}
for cylinder functions. For the Andersen Thermostat, using \eqref{eqn:example_generator_macro_anderson}, we find
\begin{multline*}
H_n F(\rho) = B_n F(\rho) + C_n F(\rho) \\
+ \gamma \int \int \left[\exp\left\{\frac{F\left(\rho - \frac{1}{n} \delta_{(q,p)} + \frac{1}{n} \delta_{(q,p')} \right) - F(\rho)}{n^{-1}}\right\} - 1 \right] \frac{1}{\sqrt{2 \pi m}} e^{- \frac{(p')^2}{2 m}} \, \dd p' \, \rho(\dd (q,p))
\end{multline*}
which has, for good differentiable functions $F$ with $\dd F \in T^*Z$, a limit
\begin{multline*}
H F(\rho) = \ip{\dd F(\rho)}{\cB^*(\rho) + \cC^*(\rho)} \\
+ \gamma \int \int \left[\exp\left\{\ip{\dd F(\rho)}{\delta_{(q,p')} -  \delta_{(q,p)}}\right\} - 1 \right] \frac{1}{\sqrt{2 \pi m}} e^{- \frac{(p')^2}{2 m}} \, \dd p' \, \rho(\dd (q,p))
\end{multline*}
Note that this expression only depends on $F$ via its functional derivative $\dd F$ at $\rho$: $Hf(\rho) = \cH(\rho,\dd F(\rho))$, where $\cH : T^*Z \rightarrow \bR$ is defined as
\begin{multline*}
\cH(\rho,\xi) = \ip{\xi}{\cB^*(\rho) + \cC^*(\rho)}  \\
+ \gamma \int \int \left[\exp\left\{ \xi(q,p') -  \xi(q,p)\right\} - 1 \right] \frac{1}{\sqrt{2 \pi m}} e^{- \frac{(p')^2}{2 m}} \, \dd p' \, \rho(\dd (q,p)).
\end{multline*}

For the Vlasov-Fokker-Planck equation, we find, using \eqref{eqn:example_generator_macro_VFP}, that
\begin{multline*}
H_n F(\rho) = \ip{\dd F(\rho)}{\cB^*(\rho) + \cC^*(\rho) + \cG^*(\rho)} + \gamma \sum_{i,j = 1}^l \partial_i\phi(\ip{\mathbf{f}}{\rho}) \partial_j \phi(\ip{\mathbf{f}}{\rho}) \ip{\nabla_p f_i \nabla_p f_j}{\rho} \\
+ \frac{\gamma}{n} \sum_{i,j = 1}^l \partial_i \partial_j \phi(\ip{\mathbf{f}}{\rho}) \ip{\nabla_p f_i \nabla_p f_j}{\rho}
\end{multline*}
with limiting $HF(\rho)$ given by
\begin{equation*}
H F(\rho) = \ip{\dd F(\rho)}{\cB^*(\rho) + \cC^*(\rho) + \cG^*(\rho)} + \gamma \vn{\dd F(\rho)}_{H_p^1(\rho)}^2,
\end{equation*}
where
\begin{equation*}
\vn{\xi}_{H_p^1(\rho)}^2 := \int \left| \nabla_p \xi(q,p) \right|^2 \dd \rho(q,p).
\end{equation*}
Again, the Hamiltonian takes the form $HF(\rho) = \cH(\rho,\dd F(\rho))$ with
\begin{equation*}
\cH(\rho,\xi) := \ip{\xi}{\cB^*(\rho) + \cC^*(\rho) + \cG^*(\rho)} + \gamma \vn{\xi}_{H_p^1(\rho)}^2 = \ip{\xi}{\bfF(\rho)} + \gamma \vn{\xi}_{H_p^1(\rho)}^2.
\end{equation*}

\subsection{Pre-generic from the large deviations}

Since the large deviations of our microscopic model are the combination of a reversible dynamics $\cG$ and deterministic dynamics $\cB + \cC$ that preserve $E$, we expect that the dissipation potentials $\Psi^*$ of Theorem \ref{theorem:AT_and_VFP_are_preGENERIC} are indeed found from the Hamiltonian $H_2$ obtained from $\cG$. We calculate $\Psi^*$ by \eqref{eqn:construction_Psi*_from_H}.

For the Andersen thermostat, we find
\begin{align*}
\Psi^*(\rho,\xi) & = \cH_2\left(\rho,\xi + \tfrac{1}{2} \dd S(\rho) \right) - \cH_2\left(\rho, \tfrac{1}{2} \dd S(\rho) \right) \\
& = \gamma \int \int \left[\exp\left\{\xi(q,p') -  \xi(q,p)\right\} - 1 \right] \frac{1}{\sqrt{2 \pi m}} e^{- \frac{(p')^2 + p^2}{2 m}} \sqrt{\rho(q,p')\rho(q,p)} \, \dd p' \, \dd p \, \dd q. 
\end{align*}
Since this expression is symmetric in $p$ and $p'$, we can write it as
\begin{equation*}
\Psi^*(\rho,\xi) = \gamma \int \int \left[\cosh\left(\xi(q,p') -  \xi(q,p)\right) - 1 \right] \frac{1}{\sqrt{2 \pi m}} e^{- \frac{(p')^2 + p^2}{2 m}} \sqrt{\rho(q,p')\rho(q,p)} \, \dd p' \, \dd p \, \dd q .
\end{equation*}
For the Vlasov-Fokker-Planck case, we have $\cH_2(\rho,\xi)= H(\rho,\xi) - \ip{\xi}{W(\rho)}$, so that
\begin{equation*}
\cH_2(\rho,\xi) = \ip{\xi}{\cG^*(\rho)} + \gamma \int |\nabla_p \xi|^2 \dd \rho = \ip{\xi}{\bfF(\rho)} + \gamma \vn{\xi}^2_{H^1_p(\rho)}.
\end{equation*}
Using that $\dd S(\rho) = -\log \rho - 1 - \dd E(\rho)$ and Lemma \ref{lemma:example_representation_of_drift_and_orthogonality}, we calculate $\Psi^*(\rho,\xi)$:
\begin{align*}
\Psi^*(\rho,\xi) & = H_2\left(\rho,\xi + \tfrac{1}{2} \dd S(\rho) \right) - H_2\left(\rho, \tfrac{1}{2} \dd S(\rho) \right) \\
& = \ip{\xi}{\cG^*(\rho)} + \gamma\left(\vn{\xi + \tfrac{1}{2}\dd S(\rho)}_{H^1_p(\rho)}^2 - \vn{\tfrac{1}{2}\dd S(\rho)}_{H^1_p(\rho)}^2 \right) \\
& = \ip{\xi}{\cG^*(\rho)} + \int \left(\gamma \nabla_p \xi \cdot \nabla_p \dd S(\rho) \right) \dd \rho  + \gamma \vn{\xi}_{H^1_p(\rho)}^2 
\end{align*}
Following the calculation in \eqref{eqn:VFP_cG_divEntropy}, we find
\begin{equation*}
\Psi^*(\rho,\xi) = \gamma \vn{\xi}_{H^1_p(\rho)}^2,
\end{equation*}
so that indeed, we reproduce the result of Theorem \ref{theorem:AT_and_VFP_are_preGENERIC}.

\subsection{Sufficient conditions for a physical extension to GENERIC}

In the sections above, the two important ingredients that are shared among the Andersen thermostat and the Vlasov-Fokker-Planck equation are the common structure of $\bfF$ and the reversibility of $\cG$ with respect to a Gaussian measure. At no point the specific form of $\cG$ is of importance, except for the calculation of $H_2$. The analysis can therefore be carried out irrespectively of the $\cG$, if the limit $H_2$ is assumed to exist.

The specific form of $\cG$ is important, however, from a physical point of view. The assumption is that $\cG$ models the interaction of the particles with the background solvent. This means that changes only occur under in the momentum variable and not in position. Assuming that the background equilibrates at a faster time scale than that of the evolution of our particles, leads to the assumption of a Gaussian stationary measure.

\begin{condition} \label{condition:reversible_wrt_gaussian}
	The vector-field $\rho \mapsto \bfF(\rho)$ is of the type \eqref{eqn:underdamped_dynamics_general} and $\cG$ and operators $G_n$ defined in terms of $\cG$ in \eqref{eqn:definition_Gn}, have the following properties:
	\begin{enumerate}[(a)]
		\item $\cG$ acts on the momentum variable only and is independent of the value of $q$,
		\item the operator $\cG$ interpreted as acting on $p$ only is reversible with respect to a centered Gaussian measure
		\begin{equation*}
		\rho_m(\dd p) = \frac{1}{\sqrt{2\pi m}} e^{-\frac{p^2}{2m}}\dd p,
		\end{equation*}
		with $m > 0$.
		\item the limit of 
		\begin{equation*}
		H_2 F(\rho) = \lim_{n \rightarrow \infty} \frac{1}{n} e^{-nF(\rho)} (G_n e^{nF})(\rho)
		\end{equation*}
		exists for cylinder functions and is of the form $H_2 F(\rho) = \cH_2(\rho,\dd F(\rho))$.
	\end{enumerate}
\end{condition}

Next, we show that the dynamics in \eqref{eqn:underdamped_dynamics_general} can be extended in two ways to GENERIC. For this we use Theorems \ref{theorem:pre_GENERIC_to_GENERIC} and \ref{theorem:pre_GENERIC_natural_variables}. The first theorem can be applied always, i.e. it is the trivial extension to GENERIC, whereas the second result uses that $L(\rho)(\dd S(\rho) - \dd E(\rho)) = 0$, see Lemma \ref{lemma:example_representation_of_drift_and_orthogonality}.

\begin{corollary}\label{corollary:underdamped_dynamics_general_extension1}
	Let Condition \ref {condition:reversible_wrt_gaussian} be satisfied. Then the dynamics
	\begin{equation} \label{eqn:underdamped_dynamics_general_extension1}
	\widehat{\bfF}(\rho,e) = \begin{bmatrix}
	- \dv_q\left(\rho \frac{p}{m}\right) + \dv_p\left( \rho \nabla_q V + \rho (\nabla_q \Phi \ast \rho)\right) + \cG^*(\rho) \\
	0
	\end{bmatrix}.
	\end{equation}
	are GENERIC for $\widehat{S}(\rho,e) = S(\rho)$, $\widehat{E}(\rho,e) = E(\rho)$, $\widehat{\Psi}^*((\rho,e),(\xi,r)) = \Psi^*(\rho,\xi)$ and 
	\begin{equation*}
	\widehat{L}(\rho,e) \begin{bmatrix}
	\xi \\ r
	\end{bmatrix} 
	=
	\begin{bmatrix}
	r L(\rho) \dd E(\rho) \\
	- \ip{L(\rho) \dd E(\rho)}{\xi}
	\end{bmatrix}.
	\end{equation*}
	
\end{corollary}

\begin{corollary} \label{corollary:underdamped_dynamics_general_extension2}
	Let Condition \ref {condition:reversible_wrt_gaussian} be satisfied. Then
	\begin{equation} \label{eqn:underdamped_dynamics_general_extension2}
	\overline{\bfF}(\rho,e) = \begin{bmatrix}
	- \dv_q\left(\rho \frac{p}{m}\right) + \dv_p\left( \rho \nabla_q V + \rho (\nabla_q \Phi \ast \rho)\right) + \cG^*(\rho) \\
	- \ip{\dd E(\rho)}{\bfF(\rho)}
	\end{bmatrix}.
	\end{equation}
	are GENERIC for $\overline{S}(\rho,e) = S(\rho) + (E(\rho) + e)$, $\overline{E}(\rho,e) = E(\rho) + e$, $\overline{\Psi}^*((\rho,e),(\xi,r)) = \Psi^*(\rho,\xi - r \dd E(\rho))$ and 
	\begin{equation*}
	\overline{L}(\rho,e) \begin{bmatrix}
	\xi \\ r
	\end{bmatrix} 
	=
	\begin{bmatrix}
	L(\rho) \xi \\
	0
	\end{bmatrix}.
	\end{equation*}
\end{corollary}

In both results, note that the different dissipation structure of the equation only affects the dissipation potential $\Psi^*$ of the GENERIC representation.

\begin{remark}
	The second result, Corollary \ref{corollary:underdamped_dynamics_general_extension2}, applied to the Vlasov-Fokker-Planck equation reproduces the result of \cite[Section 3.2]{DPZ13} after working out the effect of the quadratic potential $\Psi^*$ onto the variables $\xi$ and $r$.
\end{remark}

\section{Conclusion}

In this paper we uncovered a new connection between (a) symmetry properties of Markov processes, expressed through their large-deviation Hamiltonians, and (b) the variational structure of the limiting deterministic evolutions, expressed in terms of GENERIC properties. This connection extends the existing connection for reversible Markov processes and gradient flows, and gives a microscopic basis to the non-quadratic version of GENERIC that we use and that has been identified before (see e.g.~\cite{GrOt97,Gr15}).

In the process we also achieved a number of other advances. We gave a formulation of this non-quadratic version of GENERIC in terms of the Hamiltonians that appear naturally in large-deviation theory (Theorem~\ref{theorem:pre_GENERIC_iff_generalized_fluctuation_symmetry} and Corollary~\ref{corollary:GENERIC_iff_generalized_fluctuation_symmetry_orthogonality}). We illustrated this connection by showing how it leads to GENERIC systems for two paradigmatic examples, the Vlasov-Fokker-Planck equation and  the Andersen thermostat (Section~\ref{section:final_example_AT_and_VFP}). 

Finally, the large-deviation context identified `pre-GENERIC' as a class of systems that arises naturally from large-deviation considerations, in the same way as (generalized) gradient flows arise naturally from large-deviation properties of reversible Markov processes. 
We showed that pre-GENERIC and GENERIC are equivalent, in the sense that a system of one type can be converted into a system of the other type. For the implication GENERIC $\Longrightarrow$ pre-GENERIC this is a question of definition; for pre-GENERIC $\Longrightarrow$ GENERIC it turns out that a \emph{trivial} modification converts any pre-GENERIC system into a GENERIC system (Theorem~\ref{theorem:pre_GENERIC_to_GENERIC}). This leads to questions of physical interpretation about the relationship between the two concepts, which we leave to future publications.

\appendix

\section{Verification of the Jacobi identity in Theorem \ref{theorem:pre_GENERIC_to_GENERIC}}

We consider a state space $\widehat{Z}:=Z\times \bR$ with variables $(z,e) \in Z \times \bR$. The tangent and cotangent spaces equal $T_{z,e}\widehat{Z} = T_z Z \times \bR$ and $T_{z,e}^*\widehat{Z}  = T_z Z \times \bR$.

\begin{proposition} \label{proposition:Jacobi_pre_GENERIC_to_GENERIC}
	Let $W$ be some vector field. Denote by
	\begin{equation*}
	\widehat{L}(z,e) \begin{pmatrix}
	\xi \\ r
	\end{pmatrix} = \begin{pmatrix}
	r W(z) \\ - \ip{W(z)}{\xi}
	\end{pmatrix} = \begin{pmatrix}
	0 & W(z) \\ - W(z) & 0
	\end{pmatrix} \begin{pmatrix}
	\xi \\ r
	\end{pmatrix}.
	\end{equation*}
	The map $\widehat{L}$ satisfies the Jacobi identity. In other words, the bracket $\{F,G \} := \ip{\dd F}{\widehat{L} \dd G}$, (applied to smooth functions $F_i : Z \times \bR \rightarrow \bR$) satisfies:
	\begin{equation*}
	\{\{F_1,F_2\},F_3\} + \{\{F_3,F_1\},F_2\}  + \{\{F_2,F_3\},F_1\}  = 0.
	\end{equation*}
\end{proposition}

The proof of the proposition uses the following auxiliary lemma.

\begin{lemma} \label{lemma:Jacobi_bracket_simple_version}
	For any three pairs $(a_1,b_1),(a_2,b_2),(a_3,b_3) \in \bR^2$, denote
	\begin{equation*}
	\left[\begin{pmatrix}
	a_1 \\ b_1
	\end{pmatrix}, \begin{pmatrix}
	a_2 \\ b_2
	\end{pmatrix}, \begin{pmatrix}
	a_3 \\ b_3
	\end{pmatrix} \right] := a_2 a_3 b_1 - a_1 a_3 b_2.
	\end{equation*}
	Then we have
	\begin{equation} \label{eqn:Jacobi_bracket_simple_version}
	\left[\begin{pmatrix}
	a_1 \\ b_1
	\end{pmatrix}, \begin{pmatrix}
	a_2 \\ b_2
	\end{pmatrix}, \begin{pmatrix}
	a_3 \\ b_3
	\end{pmatrix} \right] + 
	\left[\begin{pmatrix}
	a_3 \\ b_3
	\end{pmatrix}, \begin{pmatrix}
	a_1 \\ b_1
	\end{pmatrix}, \begin{pmatrix}
	a_2 \\ b_2
	\end{pmatrix} \right]+
	\left[\begin{pmatrix}
	a_2 \\ b_2
	\end{pmatrix}, \begin{pmatrix}
	a_3 \\ b_3
	\end{pmatrix}, \begin{pmatrix}
	a_1 \\ b_1
	\end{pmatrix} \right]
	= 0.
	\end{equation}
\end{lemma}

\begin{proof}
	By definition, the left-hand side of \eqref{eqn:Jacobi_bracket_simple_version} equals
	\begin{equation*}
	a_2 a_3 b_1 - a_1 a_3 b_2 + a_1 a_2 b_3 - a_3 a_2 b_1 + a_3 a_1 b_2 - a_2 a_1 b_3 
	\end{equation*}
	which indeed equals $0$.
\end{proof}

\begin{proof}[Proof of Proposition \ref{proposition:Jacobi_pre_GENERIC_to_GENERIC}]
	The differentials $\dd F_i$ consist of two components, one on $Z$, and one on $\bR$. We denote these components by $\dd_z F_i$ and $\dd_e F_i$. We introduce the following two sets of functions $G_{ij}, H_{i} : Z \times \bR \rightarrow \bR$ defined by
	\begin{align*}
	G_{ij}(z,e) & := \ip{W(z)}{\dd_z F_i(z,e)} \cdot \dd_e F_j(z,e), \\
	H_{i}(z,e) & := \ip{W(z)}{\dd_z F_i(z,e)}.
	\end{align*}
	First note that$\{F_i,F_j\} = G_{ij} - G_{ji}$.	Next, we calculate $\dd_z$ and $\dd_e$ of $G_{ij}$:
	\begin{align*}
	\dd_z G_{ij} & = \dd_z H_i \cdot \dd_e F_j + H_i \cdot \dd_z \dd_e F_j, \\
	\dd_e G_{ij} & = \dd_e H_i \cdot \dd_e F_j + H_i \cdot \dd_e \dd_e F_j \\
	& = \ip{W}{\dd_z \dd_e F_i} \cdot \dd_e F_j + H_i \cdot \dd_e \dd_e F_j.
	\end{align*}
	
	Using these differentials, a straightforward calculation yields
	\begin{align*}
	\{\{F_i,F_j\},F_k\} & = \left\langle{\begin{pmatrix}
		\dd_z \left(G_{ij} - G_{ji}\right) \\ \dd_e \left(G_{ij} - G_{ji}\right)
		\end{pmatrix}},{\begin{pmatrix}
		\dd_e F_k \cdot W \\ - H_k
		\end{pmatrix}}\right\rangle \\
	& = \dd_e F_j \cdot \dd_e F_k \cdot \ip{W}{\dd_z H_i} - \dd_e F_i \cdot \dd_e F_k  \cdot \ip{W}{\dd_z H_j} \\
	& \qquad + H_i \cdot \dd_e F_k  \cdot  \ip{W}{\dd_z \dd_e F_j} - H_j \cdot \dd_e F_k \cdot \ip{W}{\dd_z \dd_e F_i} \\
	& \qquad - \dd_e F_j \cdot H_k \cdot \ip{W}{\dd_z \dd_e F_i} + \dd_e F_i \cdot H_k \cdot \ip{W}{\dd_z \dd_e F_j} \\
	& \qquad - H_i \cdot H_k \cdot \dd_e \dd_e F_j + H_j \cdot H_k \cdot \dd_e \dd_e F_i \\
	& = \left[\begin{pmatrix}
	 \dd_e F_i \\ \ip{W}{\dd_z H_i}
	\end{pmatrix},\begin{pmatrix}
	\dd_e F_j \\ \ip{W}{\dd_z H_j}
	\end{pmatrix},\begin{pmatrix}
	\dd_e F_k \\ \ip{W}{\dd_z H_k}
	\end{pmatrix} \right] \\
	& \qquad + \left[\begin{pmatrix}
	H_i \\ \dd_e \dd_e F_i
	\end{pmatrix},\begin{pmatrix}
	H_j \\ \dd_e \dd_e F_j
	\end{pmatrix},\begin{pmatrix}
	H_k \\ \dd_e \dd_e F_k
	\end{pmatrix} \right] \\
	& \qquad + \left( H_i \cdot \dd_e F_k + H_k \cdot \dd_e F_i\right) \cdot  \ip{W}{\dd_z \dd_e F_j} \\
	& \qquad - \left(H_j \cdot \dd_e F_k + H_k \cdot \dd_e F_j \right) \cdot \ip{W}{\dd_z \dd_e F_i}.
	\end{align*}
	It follows by Lemma \ref{lemma:Jacobi_bracket_simple_version} that
	\begin{align*}
	& \{\{F_1,F_2\},F_3\} + \{\{F_3,F_1\},F_2\}  + \{\{F_2,F_3\},F_1\} \\
	& = \left( H_1 \cdot \dd_e F_3 + H_3 \cdot \dd_e F_1\right) \cdot  \ip{W}{\dd_z \dd_e F_2} \\
	& \qquad - \left(H_2 \cdot \dd_e F_3 + H_3 \cdot \dd_e F_2 \right) \cdot \ip{W}{\dd_z \dd_e F_1} \\
	& \qquad + \left( H_3 \cdot \dd_e F_2 + H_2 \cdot \dd_e F_3\right) \cdot  \ip{W}{\dd_z \dd_e F_1} \\
	& \qquad - \left(H_1 \cdot \dd_e F_2 + H_2 \cdot \dd_e F_1 \right) \cdot \ip{W}{\dd_z \dd_e F_3} \\
	& \qquad + \left( H_2 \cdot \dd_e F_1 + H_1 \cdot \dd_e F_2\right) \cdot  \ip{W}{\dd_z \dd_e F_3} \\
	& \qquad - \left(H_3 \cdot \dd_e F_1 + H_1 \cdot \dd_e F_3 \right) \cdot \ip{W}{\dd_z \dd_e F_2} \\
	& = 0.
	\end{align*}
\end{proof}

\printbibliography


\end{document}